%% file: papier_projection.tex
\documentclass[a4paper, 10pt]{article}

\usepackage{caption}
\usepackage{subcaption}
\usepackage{amssymb,amsfonts,amsmath}
\usepackage{amsthm}   
\usepackage{verbatim}
\usepackage{epsfig}
\usepackage{mathrsfs}
\usepackage{enumerate}
\usepackage{color}
\usepackage{bbm}
\usepackage{url}
\usepackage{mathtools}
\usepackage{graphicx}

\usepackage[ruled,vlined]{algorithm2e}

\newtheorem{theorem}{Theorem}
\newtheorem{definition}{Definition}
\newtheorem{proposition}{Proposition}
\newtheorem{remark}{Remark}
\newtheorem{corollary}{Corollary}

\newcommand{\argmin}{\mathop{\mathrm{arg\, min}}}
\newcommand{\Argmin}{\mathop{\mathrm{Arg\, min}}}

\newcommand{\R}{\mathbb{R}}
\newcommand{\T}{\mathbb{T}}
\newcommand{\Z}{\mathbb{Z}}
\newcommand{\N}{\mathbb{N}}

\newcommand{\Mp}{\mathcal{M}_\Delta}
\newcommand{\Ms}{\mathcal{M}}

\newcommand{\Nk}{\mathcal{N}_{h}}

\makeatletter
\newcommand*{\rom}[1]{\expandafter\@slowromancap\romannumeral #1@}
\makeatother

\graphicspath{{images_papier/}}

\title{A projection algorithm on measures sets}

\author{Nicolas Chauffert \thanks{e-mail: nicolas.chauffert@gmail.com}, Philippe Ciuciu \thanks{e-mail: philippe.ciuciu@gmail.com}, Jonas Kahn \thanks{e-mail: jonas.kahn@math.univ-lille1.fr}, Pierre Weiss \thanks{e-mail: pierre.armand.weiss@gmail.com}\\}

\begin{document}


\maketitle

\begin{abstract}
We consider the problem of projecting a probability measure $\pi$ on a set $\mathcal{M}_N$ of Radon measures. 
The projection is defined as a solution of the following variational problem:
\begin{equation*}
\inf_{\mu\in \mathcal{M}_N} \|h\star (\mu - \pi)\|_2^2,
\end{equation*}
where $h\in L^2(\Omega)$ is a kernel, $\Omega\subset \R^d$ and $\star$ denotes the convolution operator.
To motivate and illustrate our study, we show that this problem arises naturally in various practical image rendering problems such as stippling (representing an image with $N$ dots) or continuous line drawing (representing an image with a continuous line).
We provide a necessary and sufficient condition on the sequence $(\mathcal{M}_N)_{N\in \N}$ that ensures weak convergence of the projections $(\mu^*_N)_{N\in \N}$ to $\pi$.
We then provide a numerical algorithm to solve a discretized version of the problem and show several illustrations related to computer-assisted synthesis of artistic paintings/drawings.
\end{abstract}

\section{Introduction}
\input{Introduction}

\section{Notation and preliminaries}
\label{sec:notations}
\input{Notation}

\section{Mathematical analysis}
\label{sec:math}
\input{MathAnalysis}

\section{Numerical resolution}
\label{sec:numericalresolution}
\input{Numerical_Resolution}


\section{Application to continuous line drawing}
\label{sec:continuousline}
\input{Continuous_Line_Drawing}

\section{Results}
\label{sec:results}

\input{Results}

\section{Conclusion}
\label{sec:conclusion}
\input{Conclusion}

\section*{Acknowledgements}

The authors wish to thank Gabriele Steidl for a nice presentation on halftoning which motivated the authors to work on this topic. The authors wish to thank Daniel Potts, Toni Volkmer and Gabriele Steidl for their support and help to run the excellent NFFT library \cite{keiner2009using}. They wish to thank Pierre Emmanuel Godet and Chan Hwee Chong for authorizing them to use the pictures in Figure \ref{fig:examplesart}.

\bibliographystyle{plain}
\bibliography{biblio_MeasureProjection}
\end{document}

%% file: Introduction.tex

Digital Halftoning consists of representing a grayscale image with only black and white tones~\cite{ulichney1987digital}. For example, a grayscale image can be approximated by a variable distribution of black dots with  over a white background. 
This technique, called stippling, is the cornerstone of most printing digital inkjet devices. 
A stippling result is displayed in Figure \ref{fig:StartingObservation}b. 
The lion in Figure \ref{fig:StartingObservation}a can be recognized from the dotted image shown in Figure \ref{fig:StartingObservation}b. 
This is somehow surprising since the differences between the pixel values of the two images are far from fzero.
One way to explain this phenomenon is to invoke the multiresolution feature of the human visual system~\cite{daugman1980two,pappas1999least}. 
Figures \ref{fig:StartingObservation}c and \ref{fig:StartingObservation}d are blurred versions of Figures \ref{fig:StartingObservation}a and \ref{fig:StartingObservation}b respectively. 
These blurred images correspond to low-pass versions of the original ones and are nearly impossible to distinguish. 
\begin{figure}[h]
\begin{center}
\begin{tabular}{cc}
(a)&\hspace{-.0\linewidth}(b)\\
\includegraphics[width=.4\textwidth]{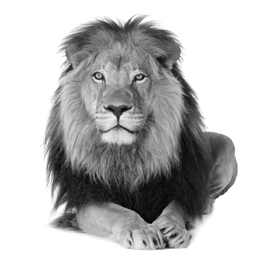}&
\includegraphics[width=.32\textwidth]{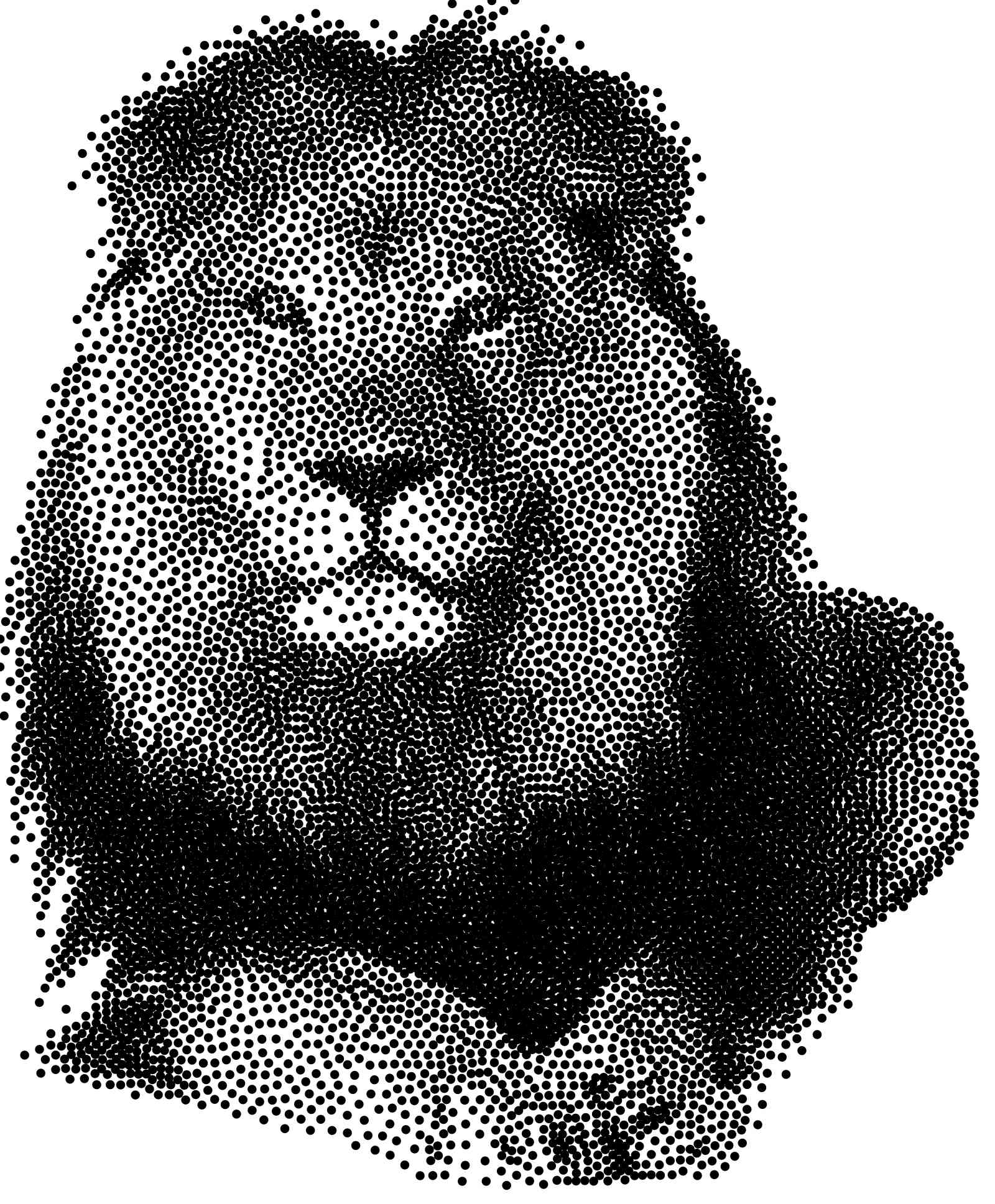}\\
(c)&\hspace{-.0\linewidth}(d)\\
\includegraphics[width=.4\textwidth]{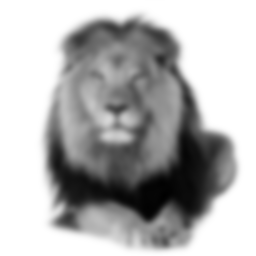}&
\hspace{-.0\linewidth}\includegraphics[width=.4\textwidth]{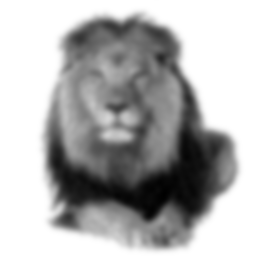}\\
\end{tabular}
\end{center}        
\caption{Explanation of the stippling phenomenon. Images (a) and (b) are similar while the norm of their difference is large. Figures (c) and (d) are obtained by convolving (a) and (b) with a Gaussian of variance equal to 3 pixels. After convolution, the images cannot be distinguished.}\label{fig:StartingObservation}
\end{figure}

Assuming that the dots correspond to Dirac masses, this experiment suggests placing the dots at locations $p_1, \hdots, p_N$ corresponding to the minimizer of the following variational problem:
\begin{equation}
\label{eq:halftone_problem}
\min_{(p_1,\hdots,p_N)\in \Omega^N} \left\|h\star \left(\pi - \frac{1}{N}\sum_{i=1}^N \delta_{p_i}\right) \right\|_2^2
\end{equation}
where $\Omega\subset \R^2$ denotes the image domain, $\delta_{p_i}$ denotes the Dirac measure at point $p_i\in \R^2$, $\pi$ denotes the target probability measure (the lion) and $h$ is a convolution kernel that should depend on the point spread function of the human visual system. 
By letting 
\begin{equation}\label{eq:NPointMeasures}
\mathcal{M}(\Omega^N)=\left\{ \mu = \frac{1}{N} \sum_{i=1}^N \delta_{p_i}, \ (p_i)_{1\leq i\leq N}\in \Omega^N\right\}
\end{equation}
denote the set of $N$-point measures, problem \eqref{eq:halftone_problem} rereads as a projection problem:
\begin{equation}
\min_{\mu \in \mathcal{M}(\Omega^N)} \left\|h\star (\pi - \mu) \right\|_2^2.
\end{equation}
This variational problem is a prototypical example that motivates our study. 
As explained later, it is intimately related to recent works on image halftoning by means of attraction-repulsion potentials proposed in \cite{schmaltz2010electrostatic,teuber2011dithering,gwosdek2014fast}. 
In references \cite{graf2012quadrature,fornasier2013consistency,fornasier2013consistency2} this principle is shown to have far reaching applications ranging from numerical integration, quantum physics, economics (optimal location of service centers) or biology (optimal population distributions). 

In this paper, we extend this variational problem by replacing $\mathcal{M}(\Omega^N)$ with an arbitrary set of measures denoted $\mathcal{M}_N$.
In other words, we want to approximate a given measure $\pi$ by another measure in the set $\mathcal{M}_N$. We develop an algorithm to perform this projection in a general setting.

To motivate this extension, we consider a practical problem: how to perform continuous line drawing with a computer?
Continuous line drawing is a starting course in all art cursus. 
It consists of drawing a picture without ever lifting the paintbrush from the page. 
Figure \ref{fig:examplesart} shows two drawings obtained with this technique. 
Apart from teaching, it is used in marketing, quilting designs, steel wire sculptures, connect the dot puzzles,...
A few algorithms were already proposed in \cite{licontinuous,xu2007image,kaplan2005tsp,bosch2004continuous,wong2011graph}.
We propose an original solution which consists of setting $\mathcal{M}_N$ as a space of pushforward measures associated with sets of parameterized curves.

Apart from the two rendering applications discussed in this paper, this paper has potential for diverse applications in fields such as imaging, finance, biology,...

\begin{figure}[h]
\begin{center}
\begin{tabular}{cc}
(a)&(b)\\
\includegraphics[width=.4\textwidth]{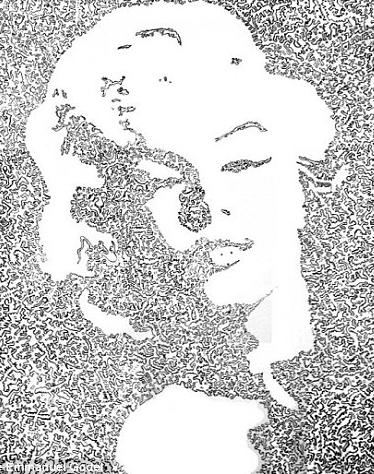}&
\includegraphics[width=.4\textwidth]{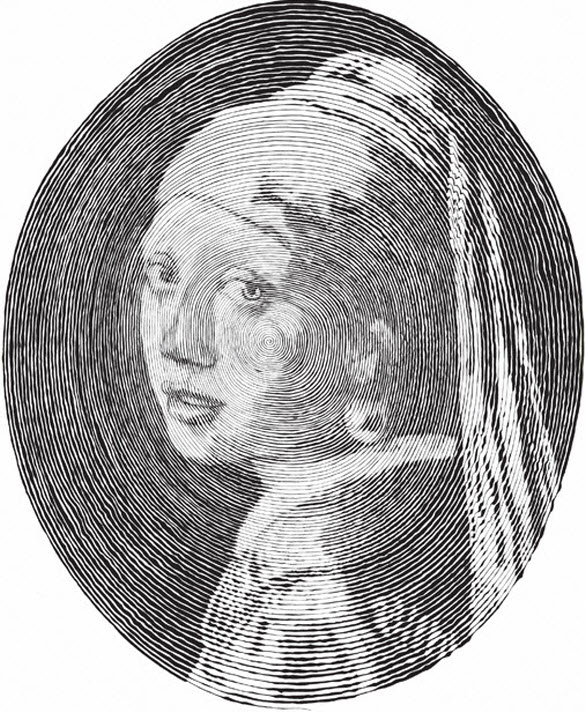}
\end{tabular}
\end{center}
\caption{Two examples of continuous line drawing. (a) A sketch of Marylin Monroe by Pierre Emmanuel Godet \protect\url{http://pagazine.com/} using a continuous line. A close inspection reveals that the line represents objects and characters. (b) Meisje met de Parel, Vermeer 1665, represented using a spiral with variable width. Realized by Chan Hwee Chong \protect\url{http://www.behance.net/Hweechong}.}\label{fig:examplesart}
\end{figure}

The remaining of this paper is structured as follows. 
We first describe the notation and some preliminary remarks in Section~\ref{sec:notations}.
We propose a mathematical analysis of the problem for generic sequences of measures spaces $(\mathcal{M}_N)_{N\in \N}$ in Section~\ref{sec:math}.
In particular, we give conditions on $h$ ensuring that the mapping $\mu \mapsto \|h\star \mu\|_2$ defines a norm on the space of signed measures and provide necessary and sufficient conditions on the sequence $(\mathcal{M}_N)_{N\in \N}$ ensuring consistency of the projection problem.
We propose a generic numerical algorithm in Section~\ref{sec:numericalresolution} and derive some of its theoretical guarantees.
In Section~\ref{sec:continuousline}, we study the particular problem of continuous line drawing from a mathematical perspective. 
Finally, we present some results in image rendering problems in Section~\ref{sec:results}.

%% file: Notation.tex
In this paper, we work on the measurable space $(\Omega,\Sigma)$, where $\Omega=\T^d$ denotes the torus $\T^d=\R^d / \Z^d$. An extension to other spaces such as $\R^d$ or $[0,1]^d$ is feasible but requires slight adaptations. Since drawing on a donut is impractical, we will set $\Omega=[0,1]^d$ in the numerical experiments.

The space of continuous functions on $\Omega$ is denoted $\mathcal{C}(\Omega)$.
The Sobolev space $(W^{m,p}([0,T]))^d$, where $m\in \N$, is the Banach space of $d$ dimensional curves in $\Omega$ with derivatives up to the $m$-th order in $L^p([0,T])$.
Let $\Mp$ denote the space of probability measures on $\Omega$, i.e. the space of nonnegative Radon measures $p$ on $\Omega$ such that $p(\Omega)=1$. 
Throughout the paper $\pi\in \Mp$ will denote a \textit{target measure}. 
Let $\Ms$ denote the space of signed measures on $\Omega$ with bounded total variation, that is $\mu = \mu_+ - \mu_-$ where $\mu_+$ and $\mu_-$ are two finite nonnegative Radon measures and $\|\mu\|_{TV} = \mu_+(\Omega ) + \mu_-(\Omega ) < \infty$.

Let $h:\Omega\to \R$ denote a continuous function. 
Let $\mu \in \Ms$ denote an arbitrary finite signed measure. 
The convolution product between $h$ and $\mu$ is defined for all $x\in \Omega$ by:
\begin{align}
 \mu \star h (x) & \coloneqq \int_{\Omega} h(x-y) d\mu(y) \label{convol}\\
				 & = \mu(h(x-\cdot)) \nonumber
\end{align}
In the Fourier space, the convolution \eqref{convol} translates to, for all $ \xi \in \mathbb{Z} ^d$~(see e.g., \cite{katznelson1968introduction}):
\begin{align*}
    \widehat{\mu \star h}(\xi) & = \hat{\mu}(\xi) \hat{h}(\xi),
\end{align*}
where $\hat \mu$ is the Fourier-Stieltjes series of $\mu$. The Fourier-Stieltjes series coefficients are defined for all $\xi\in \mathbb{Z} ^d$ by:
\begin{equation*}
\hat \mu(\xi) \coloneqq \int_{\Omega} e^{-2i\pi \langle \xi, x\rangle } \,d\mu(x).
\end{equation*}
We recall the Parseval formula:
\begin{equation*}
 \int_{\Omega} |h(x)|^2 \,dx = \sum_{\xi \in \Z^d} \left|\hat h(\xi)\right|^2.
\end{equation*}

Let $J:\R^n\to \R$ denote a function and $\partial J$ its limiting-subdifferential (or simply subdifferential)~\cite{mordukhovich2006variational,attouch2013convergence}. 
Let $C\subseteq \R^n$ denote a closed subset. 
The indicator function of $C$ is denoted $i_C$ and defined by
\begin{equation*}
i_C(x) = \left \{ \begin{array}{ll} 0 & \textrm{if } x\in C, \\ +\infty & \textrm{otherwise.}\end{array}\right.
\end{equation*}
The set of projections of a point $x_0\in \R^n$ on $C$ is denoted $P_C(x_0)$ and defined by 
\begin{equation*}
P_C(x_0) = \Argmin_{x\in C} \|x-x_0\|_2^2.
\end{equation*}
The notation $\Argmin$ stands for the whole set of minimizers while $\argmin$ denotes one of the minimizers. 
Note that $P_C$ is generally a point-to-set mapping except if $C$ is convex closed, since the projection on a closed convex set is unique. 
The normal cone at $x\in \R^n$ is denoted $N_C(x)$. It is defined as the limiting-subdifferential of $i_C$ at $x$. 
A critical point of the function $J+i_C$ is a point $x^*$ that satisfies $0\in \partial J(x^*) + N_C(x^*)$. 
This condition is necessary (but not sufficient) for $x^*$ to be a local minimizer of $J+i_C$.

%% file: MathAnalysis.tex
Let 
\begin{equation}\label{eq:defNk}
\Nk(\mu) \coloneqq \|h\star \mu\|_2.
\end{equation}
In this section, we study some basic properties of the following projection problem:
\begin{equation}\label{eq:projection}
\min_{\mu \in \mathcal{M}_N} \mathcal{N}_h(\pi - \mu),
\end{equation}
where $(\mathcal{M}_N)_{N\in \N}$ denotes an arbitrary sequence of measures sets in $\mathcal{M}_\Delta$.

\subsection{Norm properties}

We first study the properties of $\mathcal{N}_h$ on the space $\Ms$ of signed measures with bounded total variation.
The following proposition shows that it is well defined provided that $h\in \mathcal{C}(\Omega)$. 
\begin{proposition}
Let $h\in \mathcal{C}(\Omega)$ and $\mu\in \Ms$. 
Then $h\star \mu\in L^2(\Omega)$.
\end{proposition}
\begin{proof}
It suffices to remark that $\forall x\in \Omega$, $|h\star \mu(x)| \leq \|\mu\|_{TV} \|h\|_\infty<+\infty$.
Therefore, $h\star \mu \in L^\infty(\Omega)$. Since $\Omega$ is bounded, $h\in L^\infty(\Omega)$ implies that $h\in L^2(\Omega)$. 
\end{proof}

\begin{remark}\label{rem:regular}
In fact, the result holds true for weaker hypotheses on $h$. 
If $h\in  \mathcal{L}^\infty(\Omega)$, the set of bounded Borel measurable functions, $h\star \mu \in L^2(\Omega)$ since 
\begin{equation*}
\forall x\in \Omega, \ |h\star \mu(x)| \leq \|\mu\|_{TV} \left(\sup_{x\in \Omega} |h(x)|\right) <+\infty.
\end{equation*}
Note that the $L^\infty$-norm is defined with an $\mathrm{ess}\sup$ while we used a $\sup$ in the above expression. 
We stick to $h\in \mathcal{C}(\Omega)$ since this hypothesis is more usual when working with Radon measures.
\end{remark}

The following proposition gives a necessary and sufficient condition on $h$ ensuring that $\Nk$ defines a norm on $\Ms$.
\begin{proposition}
    \label{normMs}
Let $h\in \mathcal{C}(\Omega)$. 
The mapping $\Nk$ defines a norm on $\Ms$ if and only if all Fourier series coefficients $\hat{h}(\xi)$ are nonzero.
\end{proposition}
\begin{proof}
Let us assume that $\hat{h}(\xi)\neq 0, \ \forall \xi\in \Z^d$.
The triangle inequality and absolute homogeneity hold trivially.
Let us show that $\mu\neq 0 \Rightarrow \Nk(\mu)\neq 0$.
The Fourier series of a nonzero signed measure $\mu$ is nonzero, so that there is $\xi\in \Z^d$ such that $\hat{\mu}(\xi) \neq 0$. 
According to our hypothesis $\hat{h}(\xi) \neq 0$, hence $\widehat{\mu \star h}(\xi) \neq 0$ and $\Nk(\mu)\neq 0$.

On the contrary, if there exists $\xi_0\in \Z^d$ such that $\hat{h}(\xi_0)=0$. 
The non-zero measure defined through its Fourier series by 
\begin{equation*}
\hat{\mu}(\xi)=\left\{\begin{array}{ll} 1 & \textrm{if } \xi=\xi_0 \\ 0 & \textrm{otherwise}\end{array}\right.
\end{equation*}
satisfies $\Nk(\mu)=0$ and belongs to $\Ms$.
\end{proof}

From now on, owing to Proposition \ref{normMs}, we will systematically assume - sometimes without mentioning - that $h\in \mathcal{C}(\Omega)$ and that $\hat{h}(\xi)\neq 0$, $\forall \xi \in \Z^d$. Finally, we show that $\Nk$ induces the weak topology on $\Ms$. 
Let us first recall the definition of weak convergence.
\begin{definition}
A sequence of measures $(\mu_N)_{N\in \N}$ is said to weakly converge to $\mu\in \Ms$, if
\begin{equation*}
 \lim_{N\to \infty} \int_\Omega f(x)d\mu_N(x) = \int_{\Omega} f(x) d\mu(x)
\end{equation*}
for all continuous functions $f:\Omega\to \R$. The shorthand notation for weak convergence is
\begin{equation*}
\mu_N \underset{N\to \infty}{\rightharpoonup} \mu.
\end{equation*}
\end{definition}

\begin{proposition}\label{prop:weakCV}
Assume that $h\in \mathcal{C}(\Omega)$ and that $\hat{h}(\xi)\neq 0$, $\forall \xi \in \Z^d$. 
Then for all sequences $(\mu_N)_{N\in \N}$ in $\Ms$ satisfying $\|\mu_N\|_{TV}\leq M <+\infty, \ \forall N\in \N$,
\begin{equation*}
\lim_{N\to \infty }\Nk(\mu_N) =0 \quad \Leftrightarrow \quad \mu_N\underset{N\to \infty}{\rightharpoonup} 0.
\end{equation*}
\end{proposition}
\begin{proof}
Let $\left( \mu_N \right)_{N\in \mathbb{N} }$ be a sequence of signed measures in $\Ms$.

If $\mu_N \rightharpoonup 0$, then $\hat{\mu}_N(\xi) = \mu_N(e^{i2\pi \langle \xi, \cdot\rangle}) \to 0$ for all $\xi\in \Z^d$. 
Since $|\hat{\mu}_N(\xi) \hat{h}(\xi)| \leq 2 M |\hat{h}(\xi)|$ for all $\xi\in \Z^d$ and $\displaystyle \sum_{\xi\in \Z^d} |2 M \hat{h}(\xi)|^2 < \infty$, dominated convergence yields that $\Nk(\mu_N) \to 0$.

Conversely, assume that $\Nk(\mu_N) \to 0$. Since the $\mu_N$ are bounded, there are subsequences $\mu_{N_s}$ that converge weakly to a measure $\nu$ that depends on the subsequence. We have to prove that $\nu = 0$ for all such subsequences. Since $\Nk(\mu_N) \to 0$, we have $\hat{\mu}_N(\xi)  \to 0$ for all $\xi\in \Z^d$. Therefore, $\hat \nu(\xi) = 0, \ \forall \xi\in \Z^d$. This is equivalent to $\nu=0$ (see e.g. \cite[p.36]{katznelson1968introduction}), ending the proof.
\end{proof}

\subsection{Existence of solutions}
\label{sec:existence_generic}

The first important question one may ask is whether Problem \eqref{eq:projection} admits a solution or not. 
Theorem \ref{prop:existence_generic} provides sufficient conditions for existence to hold.
\begin{proposition}\label{prop:existence_generic}
If $\mathcal{M}_N$ is weakly compact, then Problem \eqref{eq:projection} admits at least a solution.
In particular, if $\mathcal{M}_N$ is weakly closed and bounded in TV-norm, Problem \eqref{eq:projection} admits at least a solution.
\end{proposition}

\begin{proof}
Assume $\mathcal{M}_N$ is weakly compact. Consider a minimizing sequence $\mu_n \in \mathcal{M}_N$. 
By compacity, there is a $\mu\in \mathcal{M}_N$ and a subsequence $(\mu_{n_k})_{k\in \N}$ such that ${\mu_{n_k}} \underset{k\to +\infty}{\rightharpoonup} \mu$. 
By Proposition~\ref{prop:weakCV}, $\mathcal{N}_h$ induces the weak topology on any TV-bounded set of signed measures, so that $\displaystyle\lim_{k\to \infty} \mathcal{N}_h(\pi - \mu_k)  = \mathcal{N}_h(\pi - \mu)$.

Since closed balls in TV-norms are weakly compact, any weakly closed TV-bounded set is weakly compact.
\end{proof}

A key concept that will appear in the continuous line drawing problem is that of pushforward or \textit{empirical measure} \cite{bogachev2007measure} defined hereafter. 
Let $(X,\gamma)$ denote an arbitrary probability space. 
Given a function $p:X\to \Omega$, the empirical measure associated with $p$ is denoted $p_*\gamma$. 
It is defined for any measurable set $B$ by
\begin{align*}
 p_*\gamma(B) & \coloneqq \gamma(p^{-1}(B)),
\end{align*}
where $\gamma$ denotes the Lebesgue measure on the interval $[0,1]$.
Intuitively, the quantity $p_*\gamma(B)$ represents the ``time'' spent by the function $p$ in $B$. 
Note that $p_*\gamma$ is a probability measure since it is positive and $p_*\gamma(\Omega)=1$.
Given a measure $\mu$ of kind $\mu=p_*\gamma$, the function $p$ is called \textit{parameterization} of $\mu$.

Let $\mathcal{P}$ denote a set of parameterizations $p:X\to \Omega$ and $\mathcal{M}(\mathcal{P})$ denote the associated set of pushforward-measures:
\begin{equation*}
\mathcal{M}(\mathcal{P}) \coloneqq \{ \mu = p_*\gamma, p\in \mathcal{P} \}.
\end{equation*}
In the rest of this paragraph we give sufficient conditions so that a projection on $\mathcal{M}(\mathcal{P})$ exists.
We first need the following proposition.
\begin{proposition}\label{prop:pointwise_weak}
Let $(p_n)_{n\in \N}$ denote a sequence in $\mathcal{P}$ that converges to $p$ pointwise. 
Then $({p_n}_*\gamma)_{n\in \N}$ converges weakly to $p_*\gamma$.
\end{proposition}
\begin{proof}
Let $f\in \mathcal{C}(\Omega)$. Since $\Omega$ is compact, $f$ is bounded. Hence dominated convergence yields
$\int_{X} f(p_n(x)) - f(p(x)) d\gamma(x) \to 0$.
\end{proof}

\begin{proposition}\label{prop:existence_pushforward}
Assume that $\mathcal{P}$ is compact for the topology of pointwise convergence. 
Then there exists a minimizer to Problem \eqref{eq:projection} with  $ \mathcal{M}_N=\mathcal{M}(\mathcal{P})$.
\end{proposition}
\begin{proof}
By Proposition \ref{prop:existence_generic} it is enough to show that $\mathcal{M}(\mathcal{P})$ is weakly compact.
First, $\mathcal{M}(\mathcal{P})$ is bounded in TV-norm since it is a subspace of probability measures.
Consider a sequence $(p_n)_{n\in \N}$ in $\mathcal{P}$ such that the sequence $({p_n}_*\gamma)_{ n\in \N}$ weakly converges to a measure $\mu$.
Since $\mathcal{P}$ is compact for the topology of pointwise convergence, there is a subsequence $(p_{n_k})_{k\in \N}$ converging pointwise to $p\in \mathcal{P}$.
By Proposition \ref{prop:pointwise_weak}, the pushforward-measure $p_*\gamma=\mu$ so that $\mu \in \mathcal{M}(\mathcal{P})$ and $\mathcal{P}$ is weakly closed.
\end{proof}

\subsection{Consistency}
\label{sec:consistency}

In this paragraph, we consider a sequence $(\mathcal{M}_N)_{N\in \N}$ of weakly compact subsets of $\mathcal{M}_\Delta$.  
By Proposition \ref{prop:existence_generic} there exists a minimizer $\mu_N^*\in \mathcal{M}_N$ to Problem \eqref{eq:projection} for every $N$.
We provide a necessary and sufficient condition on $(\mathcal{M}_N)_{N\in \N}$ for consistency, i.e. $\mu^*_N \underset{N\to \infty}{\rightharpoonup} \pi$. 
In the case of image rendering, it basically means that if $N$ is taken sufficiently large, the projection $\mu^*_N$ and the target image $\pi$ will be indistinguishable from a perceptual point of view.
The first result reads as follows.
\begin{theorem}
The following assertions are equivalent:
\begin{itemize}
\item[i)] For all $\pi\in \mathcal{M}_\Delta$, $\mu^*_N\underset{N\to \infty}{\rightharpoonup} \pi$.
\item[ii)] $\displaystyle \cup_{N\in \N} \mathcal{M}_N$ is weakly dense in $\mathcal{M}_\Delta$. 
\end{itemize}
\end{theorem}
\begin{proof}
We first prove $ii)\Rightarrow i)$. Assume that $\displaystyle \cup_{N\in \N} \mathcal{M}_N$ is weakly dense in $\mathcal{M}_\Delta$. 
This implies that, $\forall \pi\in \mathcal{M}_\Delta,\ \exists (\mu_N)_{N\in \N}\in (\mathcal{M}_N)_{N\in \N}$ such that $\mu_N\underset{N\to \infty}{\rightharpoonup} \pi$. 
From Proposition \ref{prop:weakCV}, this is equivalent to $\displaystyle \lim_{N\to  \infty}\mathcal{N}_h(\mu_N - \pi) = 0$.
Since $\mu_N^*$ is the projection 
\begin{equation*}
0 \leq \mathcal{N}_h(\mu_N^* - \pi) \leq \mathcal{N}_h(\mu_N - \pi) \to  0.
\end{equation*}
Proposition \ref{prop:weakCV} implies that $\mu^*_N\underset{N\to \infty}{\rightharpoonup} \pi$.

The proof of $i)\Rightarrow ii)$ is straightforward by contraposition. 
Indeed, if $\displaystyle \cup_{N\in \N} \mathcal{M}_N$ is not weakly dense in $\mathcal{M}_\Delta$, there exists $\pi_0 \in \mathcal{M}_\Delta$ that can not be approximated weakly by any sequence $(\mu_N)_{N\in \N}\in (\mathcal{M}_N)_{N\in \N}$. 
\end{proof}

We now turn to the more ambitious goal of assessing the speed of convergence of $\mu^*_N$ to $\pi$. The most natural metric in our context is the minimized norm $\mathcal{N}_h(\mu^*_N-\pi)$. 
However, its analysis is easy in the Fourier domain, whereas all measures sets in this paper are defined in the space domain. 
We therefore prefer to use another metrization of weak convergence, given by the transportation distance. Moreover we will see in Theorem \ref{prop:norm_comparisons} that the transportation distance defined below dominates $\mathcal{N}_h$.
\begin{definition}
The $L^1$ transportation distance, also known as Kantorovitch or Wasserstein distance, between two measures with same TV norm is given by:
\begin{align*}
    W_1(\mu, \nu) \coloneqq \inf_{c} \int \left\lVert x - y \right\rVert _1 \mathrm{d}c(x,y) 
\end{align*}
where the infimum runs over all couplings  of $\mu$ and $\nu$, that is the measures $c$ on $\Omega \times \Omega $ with marginals satisfying $c(A ,\Omega ) = \mu(A)$ and $c(\Omega , A) = \nu(A) $ for all Borelians $A$.

Equivalently, we may define the distance through the dual, that is the action on Lipschitz functions:
\begin{align}
\label{W1dual}
    W_1(\mu, \nu) = \sup_{f : Lip(f) \leq 1} \mu(f) - \nu(f). 
\end{align}
\end{definition}

We define the point-to-set distance as 
\begin{equation*}
W_1( \mathcal{M}_N, \pi) \coloneqq \inf_{\mu \in \mathcal{M}_N} W_1(\mu,\pi).
\end{equation*}
Obviously this distance satisfies:
\begin{equation}\label{eq:defdeltaN}
W_1( \mathcal{M}_N, \pi) \leq \delta_N \coloneqq \sup_{\pi \in \mathcal{M}_\Delta} \inf_{\mu \in \mathcal{M}_N} W_1(\mu,\pi).
\end{equation}

\begin{theorem}\label{prop:norm_comparisons}
Assume that $h\in \mathcal{C}(\Omega)$ denote a Lipschitz continuous function with Lipschitz constant $L$.
Then
\begin{equation}\label{eq:major_W1}
\mathcal{N}_h(\mu-\pi) \leq L W_1(\mu, \pi)
\end{equation}
and
\begin{equation}\label{eq:upbound_deltaN}
\mathcal{N}_h(\mu^*_N - \pi) \leq  L W_1( \mathcal{M}_N, \pi) \leq L \delta_N.
\end{equation}
\end{theorem}
\begin{proof}
Let $\tau_x:h(\cdot)\mapsto h(x-\cdot)$ denote the symmetrization and shift operator.
Let us first prove inequality \eqref{eq:major_W1}:
\begin{align*}
\|h\star(\mu - \pi) \|_2^2 &= \int_{\Omega} \left[ h\star(\mu - \pi) (x) \right]^2\, dx \\
&= \int_{\Omega} \left| \mu(\tau_x h)  - \pi(\tau_x h)  \right|^2 \, dx \\
&\leq |\Omega| L^2 W_1^2(\mu, \pi),
\end{align*}
where we used the dual definition~\eqref{W1dual} of the Wasserstein distance to obtain the last inequality.

Let $\mu_N$ denote a minimizer of $\displaystyle \inf_{\mu \in \mathcal{M}_N} W_1(\mu,\pi)$. 
If no minimizer exists we may take an $\epsilon$-solution with arbitrary small $\epsilon$ instead.
By definition of the projection $\mu_N^*$, we have:
\begin{equation}\label{eq:quantifW1}
\mathcal{N}_h(\mu_N^*-\pi) \leq \mathcal{N}_h(\mu_N-\pi) \leq W(\mu_N,\pi) \leq \delta_N.
\end{equation}
\end{proof}

Even though the bound \eqref{eq:upbound_deltaN} is pessimistic in general, it provides some insight on which sequences of measure spaces allow a fast weak convergence.

\subsection{Application to image stippling}

In order to illustrate the proposed theory, we first focus on the case of $N$-point measures $\mathcal{M}(\Omega^N)$ defined in Eq.~\ref{eq:NPointMeasures}.
This setting is the standard one considered for probability quantization~(see~\cite{gruber2004optimum,kloeckner2012approximation} for similar results). As mentioned earlier, it has many applications including image stippling.
Our main results read as follows.
\begin{theorem}\label{thm:approxNpoints}
Let $h$ denote an $L$-Lipschitz kernel.
The set of $N$-point measures $\mathcal{M}(\Omega^N)$ satisfies the following inequalities:
\begin{equation}\label{eq:boundNpoints}
  \delta_N = \sup_{\pi \in \mathcal{M}_\Delta} \inf_{\mu \in \mathcal{M}(\Omega^N)} W_1(\mu,\pi) \leq  \left( \frac{\sqrt{d}}{2}+1 \right) \frac{1}{N^{1/d}-1}
\end{equation}
and
\begin{equation}\label{eq:boundNpoints2}
 \sup_{\pi \in \mathcal{M}_\Delta} \inf_{\mu \in \mathcal{M}(\Omega^N)} \mathcal{N}_h(\mu - \pi) \leq  L \left( \frac{\sqrt{d}}{2}+1 \right) \frac{1}{N^{1/d}-1}.
\end{equation}
\end{theorem}

As a direct consequence, we get the following corollary.
\begin{corollary}\label{thm:consistence_N-point}
Let $\mathcal{M}_N=\mathcal{M}(\Omega^N)$ denote the set of N-point measures.
Then there exist solutions $\mu^*_N$ to the projection problem \eqref{eq:projection}.

Moreover, for any $L$-Lipschitz kernel $h\in \mathcal{C}(\Omega)$:
\begin{itemize}
\item[i)] $\mu_N^*\underset{N\to \infty}{\rightharpoonup} \pi$.
\item[ii)] $\mathcal{N}_h(\mu^*_N - \pi) = \mathcal{O}\left(  L N^{-\frac{1}{d}} \right).$
\end{itemize}
\end{corollary}
\begin{proof}
We first evaluate the bound $\delta_N$ defined in \eqref{eq:defdeltaN}.
To this end, for any given $\pi$, we construct an explicit sequence of measures $\mu_0, \dots, \mu_N$, the last of which is an $N$-point measure approximating $\pi$.

Note that $\T^d$ can be thought of as the unit cube $[0,1)^d$. 
It may therefore be partitioned in $C^d$ smaller cubes of edge length $1/C$ with $C=\lfloor N^{1/d}\rfloor$.
We let $(\omega_i)_{1\leq i \leq C^d}$ denote the small cubes and $x_i$ denote their center.
We assume that the cubes are ordered in such a way that $\omega_i$ and $\omega_{i+1}$ are contiguous.

We define $\displaystyle \mu_0 = \sum_{i=1}^{C^d} \pi(\omega_i)\delta_{x_i}$.
The measure $\mu_0$ satisfies
\begin{align*}
    W_1(\pi, \mu_0)       &  \leqslant \frac{1}{2} \sup_i \text{Diameter}(\omega _i) \\
                          &  \leqslant \frac{\sqrt{d}}{2}\lfloor N^{1/d} \rfloor^{-1} \\
			  &  \leqslant \frac{\sqrt{d}}{2}\frac{1}{N^{1/d}-1},
\end{align*}
but is not an $N$-point measure since $N \pi(\omega_i)$ is not an integer. 

To obtain an $N$-point measure, we recursively build $\mu_l$ as follows:
\begin{align*}
    \mu_l(\{x_l\}) & = \frac{1}{N} \left\lfloor{ N \mu_{l-1}(\{x_l\}) }\right\rfloor, \\
    \mu_l(\{x_{l+1}\}) & = \mu_{l-1}(\{x_{l+1}, x_l\}) 
     - \frac{1}{N} \left\lfloor{N \mu_{l-1}(\{x_l\}) }\right\rfloor \\ & \hspace{.5\linewidth} \text{\quad if } l\leq(1/C)^d-1, \\
                        \mu_l(\{x_i\}) & = \mu_{l-1}(\{x_i\}) \text{\quad if } i\notin \{l, l+1\}.    
\end{align*}
We stop the process for $l=(1/C)^d$ and let $\tilde \mu = \mu_{(1/C)^d}$.
Notice that $N\mu_l(x_i)$ is an integer for all $i\leqslant l$ and that $\mu_l$ is a probability measure for all $l$. 
Therefore $\tilde \mu$ is an $N$-point measure.
Moreover:
\begin{align*}
    W_1(\mu_l, \mu_{l+1}) & \leqslant \frac{1}{N} \|x_l - x_{l+1}\|_2 \\
			  & \leqslant \frac{1}{N(N^{1/d}-1)}.
\end{align*}
Since the transportation distance is a distance, we have the triangle inequality.
Therefore:
\begin{align*}
    W_1(\pi, \tilde \mu)       & \leq W_1(\pi, \mu_0) + \sum_{l=1}^N W_1(\mu_{l-1}, \mu_l), \\
                          &  =   \frac{\sqrt{d}}{2}\frac{1}{N^{1/d}-1} + N\frac{1}{N(N^{1/d}-1)}  \\
                          &  = \left( \frac{\sqrt{d}}{2}+1 \right) \frac{1}{N^{1/d}-1}.  
\end{align*}
The inequality \eqref{eq:boundNpoints2} is a direct consequence of this result and Proposition \ref{prop:norm_comparisons}.

We now turn to the proof of Corollary \ref{thm:consistence_N-point}.
To prove the existence, first notice that the projection problem \eqref{eq:projection} can be recast as \eqref{eq:halftone_problem}.
Let $p=(p_1,\cdots,p_N) \in \Omega^N$. 
The mapping $p\mapsto \left\|h\star \left(\pi - \frac{1}{N}\sum_{i=1}^N \delta_{p_i}\right) \right\|_2^2$ is continuous. 
Problem \eqref{eq:halftone_problem} therefore consists of minimizing a finite dimensional continuous function over a compact set. 
The existence of a solution follows.
Point ii) is a direct consequence of Theorem \ref{prop:norm_comparisons} and bound \eqref{eq:boundNpoints2}.
Point i) is due to the fact that $\mathcal{N}_h$ metrizes weak convergence, see Proposition~\ref{prop:weakCV}.
\end{proof}

%% file: Numerical_Resolution.tex
In this section, we propose a generic numerical algorithm to solve the projection problem \eqref{eq:projection}. 
We first draw a connection with the recent works on electrostatic halftoning \cite{schmaltz2010electrostatic,teuber2011dithering} in subsection \ref{sec:rel_halftone}. 
We establish a connection with Thomson's problem \cite{thomson1904xxiv} in subsection \ref{sec:Thomson}.
We then recall the algorithm proposed in \cite{schmaltz2010electrostatic,teuber2011dithering} when $\mathcal{M}_N$ is the set of $N$-point measures.
Finally, we extend this principle to arbitrary measures spaces and provide some results on their theoretical performance in section \ref{sec:generic_algo}.

\subsection{Relationship to electrostatic-halftoning}
\label{sec:rel_halftone}

In a recent series of papers \cite{schmaltz2010electrostatic,teuber2011dithering,graf2012quadrature,gwosdek2014fast}, it was suggested to use electrostatic principles to perform image halftoning. 
This technique was shown to produce results having a number of nice properties such as few visual artifacts and state-of-the-art performance when convolved with a Gaussian filter. 
Motivated by preliminary results in \cite{schmaltz2010electrostatic}, the authors of \cite{teuber2011dithering} proposed to choose the $N$ points locations $p=(p_i)_{1\leq i \leq N}\in \Omega^N$ as a solution of the following variational problem:
\begin{equation}\label{eq:electorhalftone}
 \min_{p\in \Omega^N}  \underbrace{\frac{1}{2N^2}\sum_{i=1}^N \sum_{j=1}^N H(p_i - p_j)}_{\textrm{Repulsion potential}} - \underbrace{\frac{1}{N}\sum_{i=1}^N \int_{\Omega} H(x - p_i) \,d\pi(x)}_{\textrm{Attraction potential}},
\end{equation}
where $H$ was initially defined as $H(x) = \left\|x\right\|_2$ in \cite{schmaltz2010electrostatic,teuber2011dithering} and then extended to a few other functions in \cite{graf2012quadrature}.
The attraction potential tends to attract points towards the bright regions of the image (regions where the measure $\pi$ has a large mass) whereas the repulsion potential can be regarded as a counter-balancing term that tends to maximize the distance between all pairs of points. 

Proposition \ref{prop:equivalence} below shows that this attraction-repulsion problem is actually equivalent to the projection problem \eqref{eq:projection} on the set of $N$-point measures defined in \eqref{eq:NPointMeasures}.
We let $\mathcal{P}^*$ denote the set of solutions of \eqref{eq:electorhalftone} and $\mathcal{M}(\mathcal{P}^*)=\{\mu = \frac{1}{N}\sum_{i=1}^N \delta_{p_i^*},\ p^*\in \mathcal{P}^*\}$. We also let $\mathcal{M}^*$ denote the set of solutions to problem \eqref{eq:projection}.

\begin{proposition}\label{prop:equivalence}
Let $h\in \mathcal{C}(\Omega)$ denote a kernel such that $|\hat h|(\xi) > 0, \ \forall \xi\in \Z^d$.
Define $H$ through its Fourier series by $\hat H(\xi) = |\hat h|^2(\xi)$.
Then problems \eqref{eq:projection} and \eqref{eq:electorhalftone} yield the same solutions set:
\begin{equation*}
\mathcal{M}^*=\mathcal{M}(\mathcal{P}^*).
\end{equation*}
\end{proposition}
\begin{proof}
First, note that since $H$ and $h$ are continuous both problems are well defined and admit at least one solution.
Let us first expand the $L^2$-norm in \eqref{eq:projection}:
\begin{align*}
\frac{1}{2}\|h\star (\mu - \pi) \|_2^2 & = \frac{1}{2}\langle h\star(\mu-\pi), h\star(\mu - \pi) \rangle \\
&= \frac{1}{2} \langle H\star(\mu-\pi), \mu - \pi \rangle \\
&= \frac{1}{2} \left( \langle H\star \mu, \mu\rangle - 2 \langle H\star \mu, \pi\rangle + \langle H\star \pi, \pi\rangle \right).
\end{align*}
Therefore 
\begin{align*}
\Argmin_{\mu \in \mathcal{M}_N} \frac{1}{2} \|h\star(\mu-\pi)\|_2^2 = \Argmin_{\mu \in \mathcal{M}_N} \frac{1}{2} \left( \langle H\star \mu, \mu\rangle - 2 \langle H\star \mu, \pi\rangle \right). 
\end{align*}
To conclude, it suffices to remark that for a measure $\mu$ of kind $\mu=\frac{1}{N}\sum_{i=1}^N \delta_{p_i}$, 
\begin{align*}
& \frac{1}{2} \left( \langle H\star \mu, \mu\rangle - 2 \langle H\star \mu, \pi\rangle \right) \\
 &= \frac{1}{2N^2}\sum_{i=1}^N \sum_{j=1}^N H(p_i - p_j) - \frac{1}{N}\sum_{i=1}^N \int_{\Omega} H(x - p_i) \,d\pi(x).
\end{align*}
\end{proof}

\begin{remark}
It is rather easy to show that a sufficient condition for $h$ to be continuous is that $H\in \mathcal{C}^{3}(\Omega)$ or $H$ be H\"older continuous with exponent $\alpha>2$. These conditions are however strong and exclude kernels such as $H(x) = \|x\|_2$.

From Remark \ref{rem:regular}, it is actually sufficient that $h\in \mathcal{L}^\infty(\Omega)$ for $\mathcal{N}_h$ to be well defined. 
This leads to less stringent conditions on $H$. We do not discuss this possibility further to keep the arguments simple.
\end{remark}

\begin{remark}
Corollary \ref{thm:consistence_N-point} sheds light on the approximation quality of the minimizers of attraction-repulsion functionals.
Let us mention that consistency of problem \eqref{eq:electorhalftone} was already studied in the recent papers \cite{graf2012quadrature,fornasier2013consistency,fornasier2013consistency2}. To the best of our knowledge, Corollary \ref{thm:consistence_N-point} is stronger than existing results since it yields a convergence rate and holds true under more general assumptions.
\end{remark}

Though formulations \eqref{eq:projection} and \eqref{eq:electorhalftone} are equivalent, we believe that the proposed one \eqref{eq:projection} has some advantages: it is probably more intuitive, shows that the convolution kernel $h$ should be chosen depending on physical considerations and simplifies some parts of the mathematical analysis such as consistency. 
However, the set of admissible measures $\mathcal{M}(\Omega^N)$ has a complex geometry and this formulation as such is hardly amenable to numerical implementation.
For instance, $\mathcal{M}(\Omega^N)$ is not a vector space, since adding two $N$-point measures usually leads to $(2N)$-point measures.
On the other hand, the attraction-repulsion formulation~\eqref{eq:electorhalftone} is an optimization problem of a continuous function over the set $\Omega^N$. 
It therefore looks easier to handle numerically using non-linear programming techniques. 
This is what we will implement in the next paragraphs following previous works \cite{schmaltz2010electrostatic,teuber2011dithering}.

\subsection{Link with Thomson's problem}
\label{sec:Thomson}

Before going further into the design of a numerical algorithm, let us first show that a specific instance of problem \eqref{eq:projection} is equivalent to Thomson's problem \cite{thomson1904xxiv}. This is a longstanding open problem in numerical optimization. It belongs to Smale's list of mathematical questions to solve for the \rom{21}st century \cite{smale1998mathematical}. 
A detailed presentation of Thomson's problem and its extensions is also proposed in \cite{hiriart2009new}. 

Let $\mathbb{S}=\{p \in \R^3, \|p\|_2=1\}$ denote the unit $3$-dimensional sphere. 
Thomson's problem may be enounced as follows:
\begin{equation}\label{eq:ThomsonProblem}
 \textrm{Find \ } p \in \Argmin_{(p_1,\hdots,p_N) \in \mathbb{S}^N} \sum_{i\neq j} \frac{1}{\|p_i - p_j\|_2}.
\end{equation}
The term $\sum_{i\neq j} \frac{1}{\|p_i - p_j\|_2}$ represents the electrostatic potential energy of $N$ electrons. 
Thomson's problem therefore consists of finding the minimum energy configuration of $N$ electrons on the sphere $\mathbb{S}$.

To establish the connection between \eqref{eq:projection} and \eqref{eq:ThomsonProblem}, it suffices to set $H(x)= \frac{1}{\|x\|_2}$, $\Omega=\mathbb{S}$ and $\pi=1$ in Eq.~\eqref{eq:electorhalftone}. By doing so, the attraction potential has the same value whatever the points configu\-ration and the repulsion potential exactly corresponds to the electrostatic potential. 

This simple remark shows that finding global minimizers looks too ambitious in general and we will therefore concentrate on the search of local minimizers only. 

\subsection{The case of $N$-point measures}
\label{sec:algo_Npoint}

In this section, we develop an algorithm specific to the projection on the set of $N$-point measures defined in \eqref{eq:NPointMeasures}. This algorithm generates stippling results such as in Fig.~\ref{fig:StartingObservation}. In stippling, the measure is supported by a union of discs, i.e., a sum of diracs convoluted with a disc indicator. We simply have to consider the image deconvoluted with this disc indicator as $\pi$ to include stippling in the framework of $N$-point measures.
We will generalize this algorithm to arbitrary sets of measures in the next section.
We assume without further mention that $\hat H(\xi)$ is real and positive for all $\xi$. This implies that $H$ is real and even.
Moreover, Proposition~\ref{prop:equivalence} implies that problems \eqref{eq:projection} and \eqref{eq:electorhalftone} yield the same solutions sets. We let $p=(p_1,\hdots,p_N)$ and set
\begin{equation}\label{eq:defJ}
\tilde J(p) \coloneqq \underbrace{\frac{1}{2N^2}\sum_{i=1}^N \sum_{j=1}^N H(p_i - p_j)}_{F(p)} - \underbrace{\frac{1}{N}\sum_{i=1}^N \int_{\Omega} H(x - p_i) \,d\pi(x)}_{\tilde G(p)}. 
\end{equation}
The projection problem therefore rereads as:
\begin{equation}
\min_{p\in \Omega^N} \tilde J(p).
\end{equation}
For practical purposes, the integrals in $\tilde G(p)$ first have to be replaced by numerical quadratures. 
We let $G(p)\simeq \tilde G(p)$ denote the numerical approximation of $\tilde G(p)$. 
This approximation can be written as
\begin{equation*}
G(p) = \frac{1}{N}\sum_{i=1}^N \sum_{j=1}^n w_j H(x_j-p_i)\pi_j, 
\end{equation*}
where $n$ is the number of discretization points $x_j$ and $w_j$ are weights that depend on the integration rule. In particular, since we want to approximate integration with respect to a probability measure, we require that
\begin{align*}
    \sum_{j=1}^n w_j \pi_j = 1.
\end{align*}
In our numerical experiments we use the rectangle rule. We may then take $\pi_j$ as the integral of $\pi$ over the corresponding rectangle. After discretization, the projection problem therefore rereads as:
\begin{equation}\label{eq:projection_continuous_functional}
\min_{p\in \Omega^N} J(p)\coloneqq F(p) - G(p).
\end{equation}

The following result \cite[Theorem 5.3]{attouch2013convergence} will be useful to design a convergent algorithm. 
We refer to \cite{attouch2013convergence} for a comprehensive  introduction to the definition of Kurdyka-{\L}ojasiewicz functions and to its applications to algorithmic analysis. In particular, we recall that semi-algebraic functions are Kurdyka-{\L}ojasiewicz \cite{kurdyka1998gradients}.
\begin{theorem}
Let $K : \R^n \to \R$ be $\mathcal{C}^1$ function whose gradient is $L$-Lipschitz continuous and let $C$ be a nonempty closed subset of $\R^n$. 
Being given $\varepsilon \in \left(0,\frac{1}{2L}\right)$ and a sequence of stepsizes $\gamma^{(k)}$ such that $ \varepsilon < \gamma^{(k)} < \frac{1}{L}-\varepsilon$, we consider a sequence $(x^{(k)})_{k\in \N}$ that complies with
\begin{equation}\label{algo:generic}
 x^{(k+1)}  \in P_C\left(x^{(k)} - \gamma^{(k)} \nabla K (x^{(k)}) \right), \mbox{ with } x^{(0)}\in C
\end{equation}
If the function $K+i_C$ is a Kurdyka-{\L}ojasiewicz function and if $(x^{(k)})_{k\in \N}$ is bounded, then the sequence $(x^{(k)})_{k\in \N}$ converges to a critical point $x^*$ in $C$. 
\end{theorem}

A consequence of this important result is the following.
\begin{corollary}
Assume that $H$ is a $\mathcal{C}^1$ semi-algebraic function with $L$-Lipschitz continuous gradient.
Set $0<\gamma<\frac{N}{3L}$. Then the following sequence converges to a critical point of problem \eqref{eq:projection_continuous_functional}
\begin{equation}\label{algo:Npoint}
 p^{(k+1)}  \in P_{\Omega^N}\left(p^{(k)} - \gamma \nabla J (p^{(k)}) \right), \mbox{ with } p^{(0)}\in \Omega^N.
\end{equation}

If $H$ is convex, $0<\gamma<\frac{N}{2L}$ ensures convergence to a critical point.
\end{corollary}
\begin{remark}
The semi-algebraicity is useful to obtain convergence to a critical point. 
In some cases it might however not be needed. 
For instance, in the case where $C$ is convex and closed, it is straightforward to establish the decrease of the cost function assuming only that $\nabla J$ is Lipschitz. 
Nesterov in \cite[Theorem 3]{nesterov2013gradient} also provides a convergence rate in $\mathcal{O}\left(\frac{1}{\sqrt{k+1}}\right)$ in terms of objective function values.
\end{remark}
\begin{proof}
First notice that $J$ is semi-algebraic as a finite sum of semi-algebraic functions. 

Function $J$ is $\mathcal{C}^1$ by Leibniz integral rule.
Let $\partial_k$ denote the derivative with respect to $p_k$. Then, since $H$ is even
 \begin{equation}\label{eq:sub_F}
  \partial_k F(p) = \frac{1}{N^2}\sum_{i=1}^N \nabla H(p_k-p_i)
\end{equation}
and 
\begin{equation}\label{eq:sub_G}
    \partial_k G(p) = -\frac{1}{N} \sum_{j=1}^n w_j \nabla H(x_j-p_k) \pi_j.
\end{equation}
For any two sets of $N$ points $p^{(1)}=(p_k^{(1)})_{1\leqslant k \leqslant N},\ p^{(2)}=(p_k^{(2)})_{1\leqslant k \leqslant N}$:
\begin{align*}
\|\nabla F(p^{(1)}) - \nabla F (p^{(2)})\|_2^2&=\sum_{k=1}^N \Big\|\partial_k F(p^{(1)}) - \partial_k F(p^{(2)})\Big\|_2^2 \\
&\hspace{-.3\linewidth}=\frac{1}{N^4}\sum_{k=1}^N \Big\|\sum_{i=1}^N \nabla H (p_k^{(1)}-p_i^{(1)}) - \nabla H (p_k^{(2)}-p_i^{(2)})\Big\|_2^2 \\
&\hspace{-.3\linewidth}\leqslant \frac{1}{N^4}\sum_{k=1}^N \Big( \sum_{i=1}^N L \|p_k^{(1)}-p_i^{(1)}- (p_k^{(2)}-p_i^{(2)})\|_2	\Big)^2 \\
&\hspace{-.3\linewidth}\leqslant \frac{L^2}{N^4}\sum_{k=1}^N \Big( \sum_{i=1}^N \|p_k^{(1)}- p_k^{(2)}\|_2+\|p_i^{(1)}-p_i^{(2)}\|_2\Big)^2 \\
&\hspace{-.3\linewidth}\leqslant \frac{L^2}{N^4}\sum_{k=1}^N N \Big( \sum_{i=1}^N \big(\|p_k^{(1)}- p_k^{(2)}\|_2+\|p_i^{(1)}-p_i^{(2)}\|_2\big)^2\Big) \\
&\hspace{-.3\linewidth}\leqslant \frac{2L^2}{N^3}\sum_{k=1}^N \sum_{i=1}^N \|p_k^{(1)}- p_k^{(2)}\|_2^2+\|p_i^{(1)}-p_i^{(2)}\|_2^2  \\
&\hspace{-.3\linewidth}=\frac{4L^2}{N^2}\|p^{(1)}-p^{(2)}\|_2^2,
\end{align*} 
and
\begin{align*}
\|\nabla G(p^{(1)}) - \nabla G (p^{(2)})\|_2^2&=\sum_{k=1}^N \Big\|\partial_k G(p^{(1)}) - \partial_k G(p^{(2)})\Big\|_2^2 \\
&\hspace{-.3\linewidth}=\frac{1}{N^2}\sum_{k=1}^N\Big\|\sum_{j=1}^n w_j \pi_j \big(\nabla H (p_k^{(1)}-x) - \nabla H (p_k^{(2)}-x)\big) \Big\|_2^2 \\
&\hspace{-.3\linewidth} \leqslant \frac{1}{N^2}\sum_{k=1}^N \Big(\sum_{j=1}^n w_j \pi_j  L \| p_k^{(1)}-p_k^{(2)}\| \Big)^2 \\
&\hspace{-.3\linewidth}=\frac{L^2}{N^2} \Big(\sum_{j=1}^{n} w_j \pi_j\Big) \|p^{(1)}-p^{(2)}\|_2^2 \\
&\hspace{-.3\linewidth}= \frac{L^2}{N^2} \|p^{(1)}-p^{(2)}\|_2^2.
\end{align*} 
Finally,
\begin{align*}
&\|\nabla J(p^{(1)}) - \nabla J (p^{(2)})\|_2 \\
&\leqslant \|\nabla F(p^{(1)}) - \nabla F (p^{(2)})\|_2 + \|\nabla G(p^{(1)}) - \nabla G (p^{(2)})\|_2 \\
& \leqslant \Big(\frac{2L}{N} + \frac{L}{N} \Big) \|p^{(1)}-p^{(2)}\|_2 = \frac{3L}{N} \|p^{(1)}-p^{(2)}\|_2.
\end{align*} 

Now, if we assume that $H$ is convex and $\mathcal{C}^2$ (this hypothesis is not necessary, but simplifies the proof). Then $F$ and $G$ are also convex and $\mathcal{C}^2$. We let $\nabla^2 F$ denote the Hessian matrix of $F$. Given the previous inequalities, we have $0\preccurlyeq \nabla^2 F \preccurlyeq \frac{2L}{N}\mathrm{Id}$ and $0\preccurlyeq \nabla^2 G \preccurlyeq \frac{L}{N}\mathrm{Id}$. Hence, the largest eigenvalue in magnitude of $\nabla^2(F-G)$ is bounded above by $\frac{2L}{N}$.

Moreover, the sequence $(x^{(k)})_{k\in \N}$ is bounded since $\Omega^N$ is bounded.
\end{proof}

\subsection{A generic projection algorithm}
\label{sec:generic_algo}

We now turn to the problem of finding a solution of \eqref{eq:projection}, where $\mathcal{M}_N$ denotes our arbitrary measures set.
In the previous paragraph, it was shown that critical points of $J+i_{\Omega^N}$ could be obtained with a simple projected gradient algorithm under mild assumtpions. 
Although this algorithm only yields critical points, they usually correspond to point configurations that are visually pleasing after only a few hundreds of iterations. 
For instance, the lion in Figure \ref{fig:StartingObservation}b was obtained after 200 iterations.
Motivated by this appealing numerical beha\-vior, we propose to extend this algorithm to the following abstract construction:
\begin{enumerate} 
\item Approximate $\mathcal{M}_N$ by a subset $\mathcal{A}_n$ of $n$-point measures.
\item Use the generic Algorithm \eqref{algo:generic} to obtain an approximate projection $\mu_n^*$ on $\mathcal{A}_n$. 
\item When possible, reconstruct an approximation $\mu_N\in \mathcal{M}_N$ of a projection $\mu_N^*$ using $\mu_n^*$.
\end{enumerate}

To formalize the approximation step, we need the definition of Hausdorff distance:
\begin{definition}
    \label{Hausdorff}
    The Hausdorff distance between two subsets $X$ and $Y$ of a metric space $(M, d)$ is:
    \begin{align*}
	\mathcal{H}_d(X,Y) := \max \left\{ \sup_{x\in X} \inf_{y \in Y} d(x,y), \sup_{y\in Y} \inf_{x \in X} d(y,x) \right\}.
    \end{align*}    
\end{definition}
In words, two sets are close if any point in one set is close to at least a point in the other set.
In this paper, the relevant metric space is the space of signed measures $\mathcal{M}$ with the norm $\mathcal{N}_h$. 
The corresponding Hausdorff distance is denoted $\mathcal{H}_{\mathcal{N}_h}$.

The following proposition clarifies why controlling the Hausdorff distance is relevant to design approximation sets $\mathcal{A}_n$.
\begin{proposition}
    \label{on5}
    Let $\mathcal{A}_n$ and $\mathcal{M}_N$ be two TV-bounded weakly closed sets of measures such that $\mathcal{H}_{\mathcal{N}_h}(\mathcal{A}_n, \mathcal{M}_N) \leq \varepsilon $.
    Let $\mu_n^*$ be a projection on $\mathcal{A}_n$. Then there is a point $\mu_N \in \mathcal{M}_N$ 
such that $\mathcal{N} _h(\mu_n^* - \mu_N) \leq  \varepsilon $ and $\displaystyle \mathcal{N}_h(\pi - \mu_N) \leq \inf_{\mu \in \mathcal{M}_N} \mathcal{N}_h(\pi - \mu) + 2 \varepsilon $.
\end{proposition}
\begin{corollary}
    \label{cor5}
    If $\displaystyle \lim_{n\to \infty }\mathcal{H}_{\mathcal{N}_h}(\mathcal{A}_n  ,\mathcal{M}_N) = 0$, then $(\mu_n^*)_{n\in \N}$ converges weakly along a subsequence to a solution $\mu_N^*$ of Problem \eqref{eq:projection}. 
    \end{corollary}
\begin{proof}
We first prove Proposition \ref{on5}.
Since $\mathcal{A}_n$ and $\mathcal{M} _N$ are bounded weakly closed, by Proposition \ref{prop:existence_generic}, there exists at least one projection $\mu_n^*$ on $\mathcal{A}_n$ and one projection $\mu_N^*$ on $\mathcal{M}_N$.

Moreover since $\mathcal{A}_n$ and $\mathcal{M} _N$ are bounded weakly closed, they are also closed for $\mathcal{N} _h$, so that the infimum in the Hausdorff distances are attained. Hence there exists $\mu_n\in \mathcal{A}_n$ such that $\mathcal{N}_h(\mu_n-\mu_N^*)\leq \mathcal{H}_{\mathcal{N}_h}(\mathcal{A}_n, \mathcal{M}_N)  \leq \varepsilon$ and  $\mu_N \in \mathcal{M}_N$ such that $\mathcal{N} _h(\mu_N - \mu^*_n) \leq \varepsilon$.  The proposition follows from the triangle inequality:
\begin{align*}
      \mathcal{N}_h(\mu_N -\pi) & \leq \mathcal{N} _h(\mu_N - \mu_n^* ) + \mathcal{N}_h( \mu_n^* - \pi) \\
                                    & \leq \varepsilon + \mathcal{N} _h(\mu_n - \pi) \\
                                    & \leq \varepsilon + \mathcal{N} _h(\mu_n - \mu_N^*) + \mathcal{N}_h(\mu_N^* - \pi)\\
                                    & \leq \mathcal{N}_h(\mu_N^* - \pi) + 2 \varepsilon .
\end{align*}
  
For the corollary, let us consider the sequence $(\mu^*_n)_{n\in \N}$ as $n$ tends to infinity. 
Since all $\mu_n$ are in $\mathcal{M}_\Delta$, which is weakly compact, we have a subsequence that converges to $\mu_\infty^*$. 
Since $\mathcal{N}_h$ is a metrization of weak convergence on $\mathcal{M}_N$, this $\mu_\infty^*$ is indeed a solution to Problem \eqref{eq:projection}:
\begin{align*}
    \mathcal{N}_h(\mu_\infty^* - \pi) &  = \lim_{n\to \infty} \mathcal{N}_h(\mu_n^* - \pi) \\
                               &  = \inf_{\mu \in \mathcal{M}_N} \mathcal{N}_h(\pi - \mu).
\end{align*}
\end{proof}

To conclude this section, we show that it is always possible to construct an approximation set $\mathcal{A}_n\subseteq \mathcal{M}(\Omega^n)$ with a control on the Hausdorff distance to $\mathcal{M}_N$. 
Let $\mathcal{M}_N^\epsilon$ denote an $\epsilon$-enlargement of $\mathcal{M}_N$ w.r.t. the $\mathcal{N}_h$-norm, i.e.: 
\begin{equation}
\mathcal{M}_N^\epsilon=\cup_{\mu_N \in \mathcal{M}_N} \{\mu \in \mathcal{M}_\Delta, \mathcal{N}_h(\mu-\mu_N) \leq \epsilon\}.
\end{equation}
We may define an approximation set $\mathcal{A}_n^\epsilon$ as follows:
\begin{equation}\label{eq:defAn}
\mathcal{A}_n^\epsilon=\mathcal{M}(\Omega^n) \cap \mathcal{M}_N^\epsilon.
\end{equation}
For sufficient large $n$, this set is non-empty and can be rewritten as
\begin{equation}\label{eq:defAn_inter}
\mathcal{A}_n^\epsilon=\left\{\mu = \frac{1}{n}\sum_{i=1}^n \delta_{p_i}, \ \textrm{ with } p=(p_i)_{1 \leq i \leq n} \in \mathcal{P}_n^\epsilon\right\},
\end{equation}
where the parameterization set $\mathcal{P}_n^\epsilon$ depends on $\mathcal{M}_N$ and $\epsilon$. 
With this discretization of $\mathcal{M}_N$ at hand, one can then apply (at least formally) the following projected gradient descent algorithm:
\begin{equation}\label{algo:npoint_appl}
 p^{(k+1)}  \in P_{\mathcal{P}_n^\epsilon}\left(p^{(k)} - \gamma \nabla J (p^{(k)}) \right), \mbox{ with } p^{(0)}\in \mathcal{P}^\epsilon_n.
\end{equation}

The following proposition summarizes the main approximation result:
\begin{proposition}
Assume that $h$ is $L$-Lipschitz.
Set $\epsilon = \left( \frac{\sqrt{d}}{2} + 1 \right)\frac{L}{n^{1/d}-1}$ and $\mathcal{A}_n = \mathcal{A}_n^{\epsilon}$, then
\begin{equation*}
\mathcal{H}_{\mathcal{N}_h}\left( \mathcal{A}_n , \mathcal{M}_N\right) = \mathcal{O}\left(Ln^{-1/d}\right). 
\end{equation*}
\end{proposition}
\begin{proof}
By construction, $\mathcal{A}_n$ satisfies
\begin{equation*}
\sup_{\mu_n\in \mathcal{A}_n} \inf_{\mu_N\in \mathcal{M}_N} \mathcal{N}_h(\mu_n-\mu_N) \leq \epsilon.
\end{equation*}

Let $\mu_N$ be an arbitrary measure in $\mathcal{M}_N$. 
By inequality \eqref{eq:boundNpoints}, there exists $\mu_n\in \mathcal{M}(\Omega^n)$ such that $\mathcal{N}_h(\mu_n-\mu_N)\leq \epsilon$.
Therefore $\mu_n$ also belongs to $\mathcal{A}_n^\epsilon$. This shows that
\begin{equation*}
\sup_{\mu_N\in \mathcal{M}_N} \inf_{\mu_n\in \mathcal{A}_n} \mathcal{N}_h(\mu_n-\mu_N) \leq \epsilon.
\end{equation*}
\end{proof}

The approximation process proposed \eqref{eq:defAn} is non-constructive in the does not induce any explicit formula for $\mathcal{P}_n^\epsilon$. 
Moreover, $\mathcal{P}^\epsilon_n$ can be an arbitrary set and the projection on $\mathcal{P}^\epsilon_n$ might not be implementable. 
We will provide constructive approximations for specific measures spaces in Section \ref{sec:continuousline}.

%% file: Continuous_Line_Drawing.tex
In this section, we concentrate on the continuous line drawing problem described in the introduction. 
We first construct a set of admissible measures $\mathcal{M}_T$ that is a natural representative of artistic continuous line drawings. 
The index $T$ represents the time spent to draw the picture.
We then show that using this set in problem \eqref{eq:projection} ensures existence of a solution and weak convergence of the minimizers $\mu_{T}^*$ to any $\pi  \in \mathcal{M}_\Delta$.
We finish by designing a numerical algorithm to solve the problem and analyze its theoretical guarantees.

\subsection{Problem formalization}

Let us assume that an artist draws a picture with a pencil. 
The trajectory of the pencil tip can be defined as a parameterized curve $p:[0,T]\to \Omega$.
The body, elbow, arm and hand are subject to non-trivial constraints \cite{marteniuk1987constraints}.
The curve $p$ should therefore belong to some admissible parameterized curves set denoted $\mathcal{P}_T$.
In this paper, we simply assume that $\mathcal{P}_T$ contains curves with bounded first and second order derivatives in $L^q([0,T])$. 
More precisely, we consider the following sets of admissible curves:
\begin{enumerate}
 \item Curves with bounded speed:
\begin{align*}
\mathcal{P}_T^{1,\infty} = \Big\{ p\in (W^{1,\infty}([0,T]))^d,\  p([0,T]) \subset \Omega, \|\dot p\|_\infty \leq \alpha_1\Big\},
\end{align*}
where $\alpha_1$ is a positive real.
 \item Curves with bounded first and second-order derivatives:
\begin{align*}
\mathcal{P}_T^{2,\infty} = \Big\{ p\in (W^{2,\infty}([0,T]))^d,\ p([0,T]) \subset \Omega, \|\dot p\|_\infty \leq \alpha_1,&\\
 \ \|\ddot p\|_\infty \leq \alpha_2 \Big\},&
\end{align*}
where $\alpha_1$ and $\alpha_2$ are positive reals. 
This set models rather accurately kinematic constraints that are met in vehicles.
It is obviously a rough approximation of arm constraints. 
 \item The proposed theory and algorithm apply to a more general setting. 
For instance they cover the case of curves with derivatives up to an arbitrary order bounded in $L^q$ with $q\in [1,\infty]$. We let
\begin{align*}
\mathcal{P}_T^{m,q}  = \Big\{ p\in (W^{m,q}([0,T]))^d,\ p([0,T]) \subset \Omega, &\\
 \forall i\in \{1,\hdots,m\},\ \|p^{(i)}\|_q \leq \alpha_i\Big\}.&
\end{align*}
where $(\alpha_i)_{i=1\hdots m}$ are positive reals. 
This case will be treated only in the numerical experiments to illustrate the variety of results that can be obtained in applications.
\end{enumerate}
Note that all above mentionned sets are convex. The convexity property will help deriving efficient numerical procedures.


In the rest of this section, we consider the following projection problem:
\begin{equation}\label{eq:projection_continuous_line}
\inf_{\mu \in \mathcal{M}\left( \mathcal{P}_T^{m,q} \right)} \mathcal{N}_h(\mu-\pi),
\end{equation}
with a special emphasis on the set $\mathcal{M}\left( \mathcal{P}_T^{m,\infty}\right)$ since it best describes standard kinematic constraints.
This problem basically consists of finding the ``best'' way to represent a picture in a given amount of time $T$.

\subsection{Existence and consistency}

We first provide existence results using the results derived in Section \ref{sec:math} for $q=\infty$. 
\begin{theorem}
For any $m\in \N^*$, Problem \eqref{eq:projection_continuous_line} admits at least one solution in $\mathcal{M}\left( \mathcal{P}_T^{m,\infty} \right)$.
\end{theorem}
\begin{proof}
From Proposition \ref{prop:existence_pushforward}, it suffices to show that $\mathcal{P}_T^{m,\infty}$ is compact for the topology of pointwise convergence.

Let $(p_n)_{n\in \N}$ be a sequence in  $\mathcal{P}_T^{m,\infty}$  that converges pointwise to $p$.     
Since $p_n$ is in $W^{m,\infty}$, its $(m-1)$-th derivative is Lipschitz continuous. By definition of $\mathcal{P}_T^{m,\infty}$, the $p^{(m-1)}_n$ are both uniformly bounded by $\alpha _{m-1}$ and $\alpha _m$-Lipschitz, hence equicontinuous. Next, by Ascoli's theorem, up to taking a subsequence, $p_n^{(m-1)}$ uniformly converges to a continuous $p^{(m-1)}$. Integrating yields that $p^{(i)}_n \to p^{(i)}$ uniformly for all $i \leq m-1$, so that $\left\lVert p^{(i)} \right\rVert _{\infty} \leq \alpha _i$ for $i \leq m-1$. Finally, a limit of $L$-Lipschitz functions is also $L$-Lipschitz, so that $\left\lVert p^{(m)} \right\rVert _{\infty} \leq \alpha _m$. Hence $p \in \mathcal{P}_T^{m,\infty}$, ending the proof.
\end{proof}

Let us now turn to weak convergence.
\begin{theorem}\label{thm:consistent}
Let $T$ be an arbitrary positive real. 
Let $\mu_T^* \in \mathcal{M}\left(\mathcal{P}_T^{m,\infty}\right)$ denote any solution of Problem \eqref{eq:projection_continuous_line}. 
Then, for any Lipschitz kernel $h\in \mathcal{C}(\Omega)$: 
\begin{itemize}
\item[i)] $\mu_T^* \underset{T\to \infty}{\rightharpoonup} \pi$,
\item[ii)] $\mathcal{N}_h(\mu^*_T - \pi) = \mathcal{O}\left(  T^{-\frac{m}{m(d+1)-1}} \right)$.
\end{itemize}
\end{theorem}

\begin{proof}
Let us consider a function $u: [0,1] \to \mathbb{R} $ such that:
\begin{itemize}
    \item{The $m$-th derivative is bounded by $\alpha _m$, that is $\left\lVert  u^{(m)} \right\rVert _{\infty} \leq \alpha _m$.}
    \item{For all integers $i \in \{1, \hdots,m-1\}$, endpoint values are zero,  that is $u^{(i)}(0) = u^{(i)}(1) = 0$.}
    \item{Start point is zero, that is $u(0) = 0$.}
    \item{Endpoint is positive, that is $u(1) = C > 0$.}
\end{itemize}

Let $x$ and $y$ in $\Omega $, such that $\|x-y\|_2 = C r^m$, and let $\tau_{xy} $ be the unit vector from $x$ to $y$. Then, for $r$ small enough,  the function $s[x,y]: t \mapsto x + \tau_{xy} u(\frac{t}{r})$ belongs to $\mathcal{P}_T^{m,\infty}$, with all its first $(m-1)$ derivatives zero at its endpoints. The condition $r$ small enough is for controlling the norm of the $i$-th derivatives for $i \leq m-1$, which scale as $r^{m-i}$.

Now, let us split $\Omega = [0,1]^d$ in $N^d$ small cubes $\omega _i$. We may order them such that each $\omega _i$ is adjacent to the next cube $\omega _{i+1}$. We write $x_{i}$ for the center of $\omega _i$. We now build functions $s \in \mathcal{P}_T^{m,\infty}$ by concatenating paths from $x_i$ to $x_{i+1}$ and waiting times in $x_i$:
\begin{align*}
    & 0  =  t_1^1  \leq \dots \leq t_{i-1}^2 \leq t_i^1 \leq t_i^2 \leq t_{i+1}^1 \leq \dots \leq t_{N^d}^2 = T, \\
    & t_i^2 - t_i^1 = \left( \frac{1}{NC}\right) ^{\frac{1}{m} }  , \\
    & s(t) = \left\{
          \begin{array}{ll}
            x_i & \qquad \text{if}\quad t_i^1 \leq t \leq t_i^2, \\
              s[x_i,x_{i+1}](t - t_i^2) & \qquad \text{if}\quad t_i^2 \leq t \leq t_{i+1}^1, \\
          \end{array}
          \right.
\end{align*}
under the condition $T \geq T_N \coloneqq (N^d - 1) \left( \frac{1}{NC}\right) ^{\frac{1}{m} }$, that is to say that we have enough time to loop through all the cube centers.

Let now $\pi \in \mathcal{M} _{\Delta}$. We may choose $t_i^2 - t_i^1 \leq T\pi(\omega _i)$ for all $i$. Then, we may couple $\pi$ and $s_* \gamma _T$ with $c(x_i, \omega _i) = \frac{t_i^2 - t_i^1}{T}$. Since the small cubes have radius $\sqrt{d}/N$ and the big one has radius $\sqrt{d}$, we obtain:
\begin{align*}
    W_1 (\pi , s_{*} \gamma _T) & \leq \frac{\sqrt{d}}{2N} \sum_i \frac{t_i^2 - t_i^1}{T} + \sqrt{d} \sum_{i < N^d}  \frac{t_{i+1}^1 - t_i^2}{T} \\
                                    & =     \frac{\sqrt{d}}{2N} \frac{T - T_N}{T} + \sqrt{d} \frac{T_N}{T} .
\end{align*}
In particular, taking $N = T^{\frac{m}{m(d+1) - 1} }$, we find that $W_1\left(\mathcal{M}\left(\mathcal{P}_T^{m, \infty}\right) , \pi\right) = O \left(  T^{- \frac{m}{m(d+1) - 1} } \right) $, hence $\bigcup_T \mathcal{M}\left(\mathcal{P}_T^{m, \infty}\right)$ is weakly dense in $\mathcal{M}_\Delta $.

\end{proof}

\subsection{Numerical resolution}

We now turn to the numerical resolution of problem \eqref{eq:projection_continuous_line}.
We first discretize the problem.
We set $\Delta t:= \frac{T}{N}$ and define discrete curves $s$ as vectors of $\R^{N\cdot d}$. 
We let $s(i)\in \R^d$ denote the curve location at discrete time $i$, corresponding to the continuous time $i\Delta t$. 

We define $D_1:\R^{N\cdot d} \to \R^{N\cdot d}$, the discrete first order derivative operator, as follows:
\begin{align*}
(D_1 s) (i) = & \frac{1}{\Delta t}\left\{
\begin{array}{ll}
0& \mbox{if } i=1, \\
s(i)-s(i-1) & \mbox{if } i\in \{2, \hdots , N\}. \\
\end{array}
\right.
\end{align*} 
In what follows, $D_i$ denotes a discretization of the derivative operator of order $i$. In the numerical experiments, we set $D_2=-D_1^*D_1$.

We define $P_N^{m,q}$, a discretized version of $\mathcal{P}_T^{m,q}$, as follows:
\begin{align}
P_N^{m,q}=\big\{s \in \R^{N\cdot d}, \mbox{such that } \forall i\in \{1,\hdots N\},\ s(i) \in \Omega, \\
\mbox{and } \forall j\in \{1, \hdots, m \}, \ \|D_j s \|_q \leqslant \alpha_j \big\}.
\end{align}
Here, $\|\cdot \|_q$ is defined by: $\displaystyle \|x\|_q=\left( \sum_{i=1}^{N\cdot d}\|x_i\|_2^q\right)^{\frac{1}{q}}$ for $q\in [1,+\infty)$ and $\displaystyle \|x\|_\infty=\max_{1\leqslant i\leqslant N\cdot d} \|x_i\|_2$.

The measures set $\mathcal{M}(\mathcal{P}_T^{m,q})$ can be approximated by the set of $N$-point measures $\mathcal{M}(P_N^{m,q})$.
From Corollary  \ref{cor5}, it suffices to control the Hausdorff distance $\mathcal{H}_{W_1}(\mathcal{M}(\mathcal{P}_T^{m,q}), \mathcal{M}(P_N^{m,q}))$, to ensure that the solution of the discrete problem~\eqref{eq:projection} with $\mathcal{M}_N=\mathcal{M}(P_N^{m,q})$ is a good approximation of problem~\eqref{eq:projection_continuous_line}. 
Unfortunately, the control of this distance is rather technical and falls beyond the scope of this paper for general $m$ and $q$. 
In the following proposition, we therefore limit ourselves to the case $m=1, q=\infty$. 

\begin{proposition}
    \label{approx_curves}
%
$\mathcal{H}_{W_1}(\mathcal{M}(\mathcal{P}_T^{1,\infty}), \mathcal{M}(P_N^{1,\infty})) \leqslant \alpha_1\frac{T}{N}$.
\end{proposition}
\begin{proof}
\begin{enumerate}
\item Let us show that $\displaystyle \sup_{\mu \in \mathcal{M}(\mathcal{P}_T^{1,\infty})} \inf_{\tilde \mu \in \mathcal{M}(P_N^{1,\infty})} W_1(\mu, \tilde \mu) \leqslant \frac{\alpha_1 T}{N}$.\\
Let $\mu \in \mathcal{M}(\mathcal{P}_T^{1,\infty})$ and denote by $p \in \mathcal{P}_T^{1,\infty}$ a parameterization such that $\mu=p_* \gamma$.
Define $\displaystyle \tilde{\mu}=\frac{1}{N} \sum_{i=0}^{N-1} \delta_{p\left(\frac{i T}{N}\right)}$. Then a parameterization of $\tilde{\mu}$ is defined by $s(i)=p\left(\frac{iT}{N}\right)$. Moreover, for $i\in \{2, \hdots N\}$, $\displaystyle |(D_1 s)(i)|=\frac{1}{\Delta t}\left|p\left(\frac{iT}{N}\right)-p\left(\frac{(i-1)T}{N}\right)\right|= \frac{1}{\Delta t}\left|\int_{\frac{(i-1)T}{N}}^{\frac{iT}{N}}\dot{p}(t)\,dt\right| \leqslant \frac{1}{\Delta t}\int_{\frac{(i-1)T}{N}}^{\frac{iT}{N}}|\dot{p}(t)|\,dt \leqslant\alpha_1$. Therefore $s\in P_N^{1,\infty}$.\\

Let us consider the transportation map coupling the curve arcs between times $(i-1)\frac{T}{N}$ and $i\frac{T}{N}$ and the Diracs at $p\left(i\frac{T}{N}\right)$. Then $\displaystyle W_1(p_*\gamma,s_*\gamma)\leqslant \sum_{i=1}^N \frac{1}{N} \sup_{(i-1)\frac{T}{N} \leqslant t \leqslant i\frac{T}{N}}\left\|s(t)-s\left((i-1)\frac{T}{N}\right)\right\| \leqslant \alpha_1 \frac{T}{N}$.

\item 
 
Let us fix $\mu\in \mathcal{M}\left(P_N^{1,\infty}\right)$ and let $s\in P_N^{1,\infty}$ such that $s_*\gamma = \mu$. We set $p(0)=s(1)$, and:
\begin{align*}
&p(t)\! =\! &\hspace{-.05\linewidth} \left\{\!
\begin{array}{ll}
\!s(1)&\hspace{-.04\linewidth}  \mbox{for } t\in \left]0, \frac{T}{N}\right], \\
\!s(i)\!+\!\left(\frac{t}{\Delta t}-i\right)\left(s(i+1)\!-\!s(i)\right) & \hspace{-.04\linewidth}\mbox{for } t\in \left]\frac{iT}{N}, \frac{(i+1)T}{N}\right]
\end{array}
\right. \\
& &\hspace{-.05\linewidth},  i\in\{1,\hdots\, N-1\}. 
\end{align*} 
Since $s \in \Omega^N$ and $\Omega$ is convex, $p([0,T])\subset \Omega$. Moreover, $p$ is continuous and piecewise differentiable. Finally, for $i\in\{1,\hdots, N-1\}$ and $t\in \left]\frac{iT}{N}, \frac{(i+1)T}{N}\right]$, $\dot{p}(t)=\frac{1}{\Delta t}\left(s(i+1)-s(i)\right)=D_1(s)(i)$. Therefore, $\|\dot p\|_\infty \leqslant \alpha_1$, ensuring that $p \in \mathcal{P}_T^{1,\infty}$.
With the same coupling as above, we have  $W_1(p_*\gamma,s_*\gamma)\leqslant  \alpha_1 \frac{T}{N}$, which ends the proof.
\end{enumerate}   
\end{proof}

To end up, let us describe precisely a solver for the following variational problem:
\begin{equation}\label{eq:projection_continuous_line_W1inf}
\inf_{\mu \in \mathcal{M}\left( \mathcal{P}_T^{1,\infty} \right)} \mathcal{N}_h(\mu-\pi).
\end{equation}
We let $\mathcal{M}^*$ denote the set of minimizers and $\mathcal{P}^*$ denote the associated set of parameterizations. 

\begin{algorithm}[h]
\KwIn{
\begin{itemize}
 \item[-] $\pi$: target measure. 
 \item[-] $N$: a number of discretization points.
 \item[-] $s^{(0)}\in P_N^{1,\infty}$: initial parameterized curve. 
 \item[-] $H$: a semi-algebraic function with Lipschitz continuous gradient. 
 \item[-] $nit$: number of iterations. 
\end{itemize}
}
\KwOut{
\begin{itemize}
 \item[-] $s^{(nit)}$: an approximation of a curve in $\mathcal{P}^*$.
 \item[-] $\mu^{(nit)}=(s^{(nit)})_*\gamma_T$: an approximation of an element of $\mathcal{M}^*$.
\end{itemize}
}
\For{$0\leq k \leq nit$}{
\begin{itemize}
 \item[-] Compute $\eta^{(k)}=\nabla J(s^{(k)})$ \\
 \item[-] Set $s^{(k+1)} = P_{P_N^{1,\infty}} \left( s^{(k)} - \tau \eta^{(k)}\right)$ \\
\end{itemize}
}
\caption{A projection algorithm on $\mathcal{M}\left(\mathcal{P}_T^{1,\infty}\right)$.\label{algo}}
\end{algorithm}

\begin{remark}
The implementation of Algorithm \ref{algo} requires computing the gradients \eqref{eq:sub_F} and \eqref{eq:sub_G} and computing a projection on $P_N^{1,\infty}$.
Both problems are actually non trivial.

The naive approach to compute the gradient of $F$ consists of using the explicit formula \eqref{eq:sub_F}. 
This approach is feasible only for a small amount of points $N$ (less than $1000$) since its complexity is $\mathcal{O} \left(N^2\right)$.
In our numerical experiments, we therefore resort to fast summation algorithms \cite{potts2003fast,keiner2009using} commonly used in particles simulation.
This part of the numerical analysis is described in \cite{teuber2011dithering} and we do not discuss it in this paper.

The set $P_N^{1,\infty}$ and more generally the sets $P_N^{m,q}$ are convex for $q\in [1,\infty]$. Projections can be computed using first-order iterative algorithms for convex functions. In our numerical experiments, we use accelerated proximal gradient descents on the dual problem \cite{beck2009fast,nesterov2013gradient,weiss2009efficient}. A precise description is given in~\cite{chauffert2014gradient}.
\end{remark}



%% file: Results.tex
To illustrate the results, we focus on the continuous line drawing problem discussed throughout the paper.
It is performed using Algorithm~\ref{algo}. 
In the following experiments, we set $H$ as a smoothed $L^1$-norm. 
This is similar to what was proposed in the original halftoning papers in \cite{schmaltz2010electrostatic,teuber2011dithering}.

\subsection{Projection onto $P_N^{1,\infty}$}

In this part, we limit ourselves to the projection onto $P_N^{1,\infty}$ as studied in the previous section. In Figure~\ref{fig:evolution}, we show the evolution of the curve $s^{(k)}$ across iterations, for different choices of $s^{(0)}$. After $30,000$ iterations, the evolution seems to be stabilized. The cost function during the 400 first iterations is depicted in Figure~\ref{fig:decayCF} for the three different initializations.

\begin{figure}[h!]
\begin{tabular}{cccc}

\rotatebox{90}{\hspace{.15\linewidth}$s^{(0)}$}&
\hspace{-.08\linewidth}\includegraphics[width=.38\linewidth]{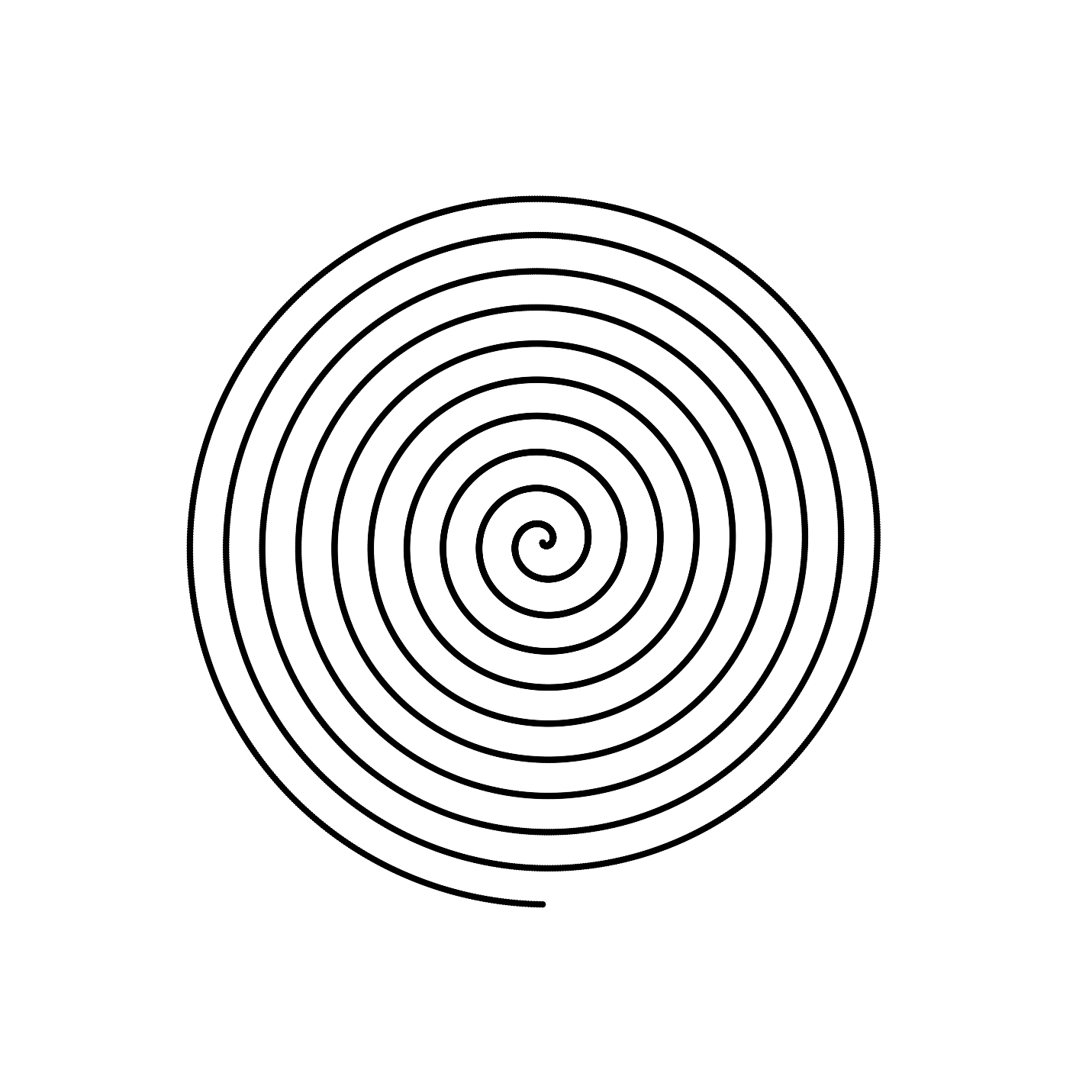}&
\hspace{-.11\linewidth}\includegraphics[width=.38\linewidth]{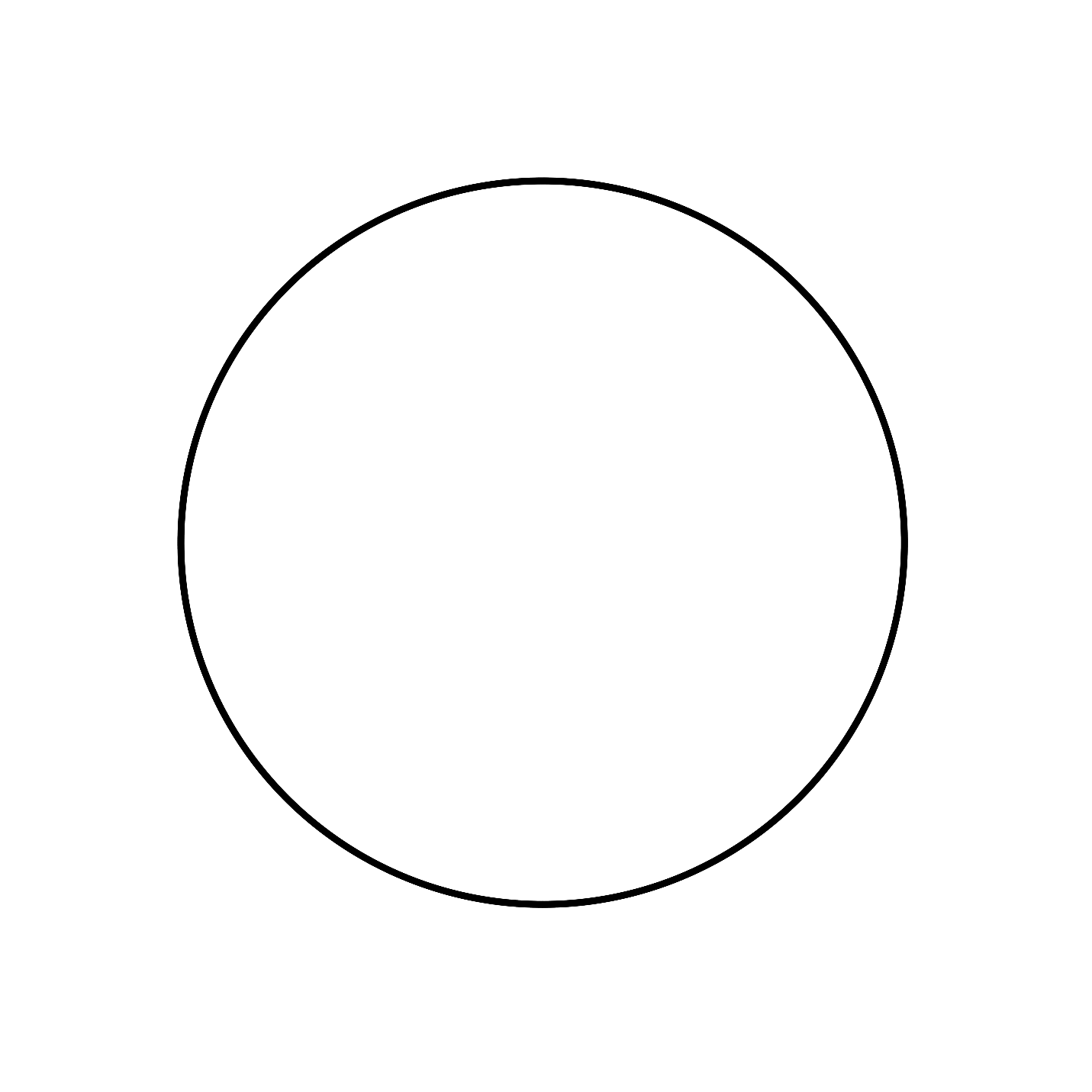}&
\hspace{-.11\linewidth}\includegraphics[width=.38\linewidth]{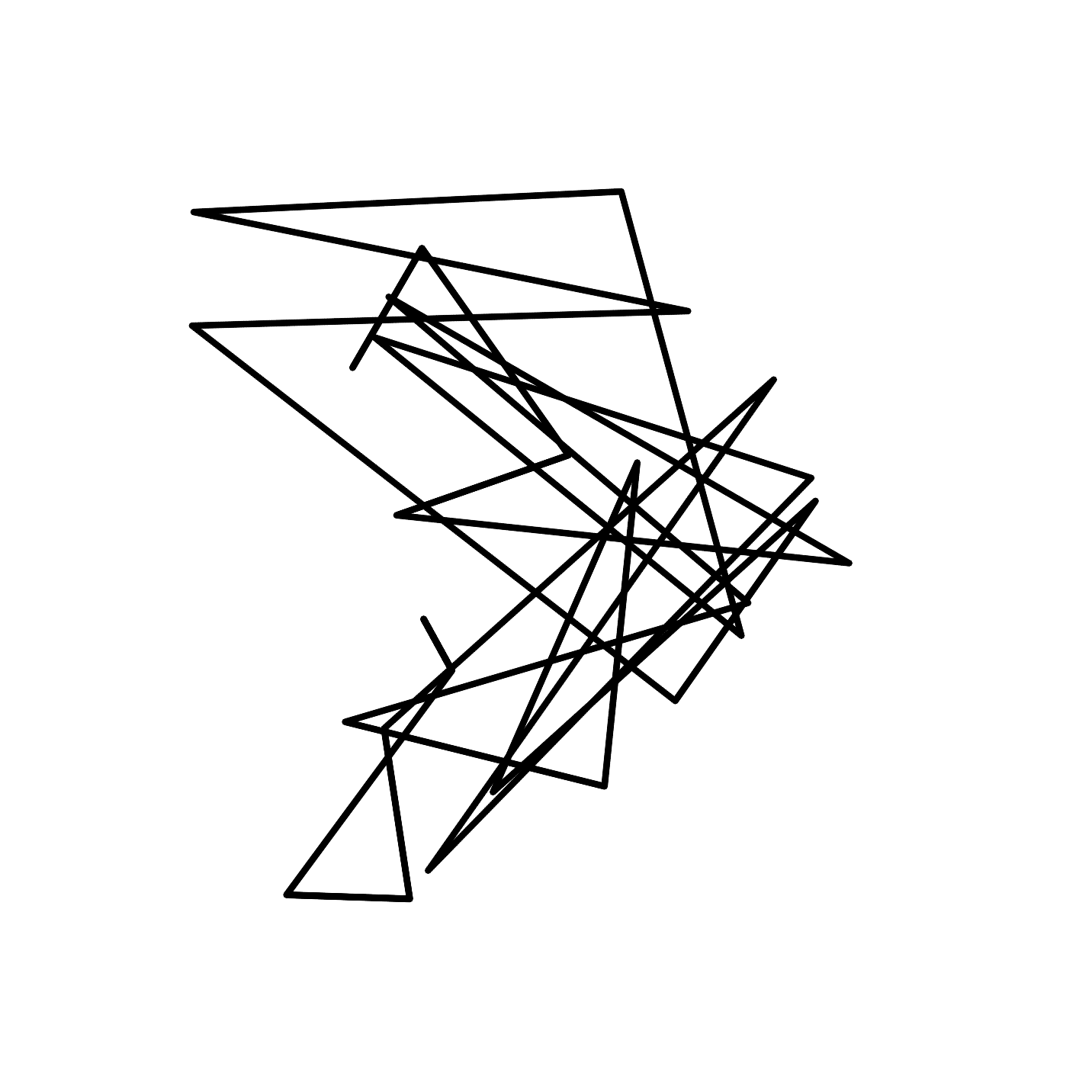}\\
\rotatebox{90}{\hspace{.15\linewidth}$s^{(100)}$}&
\hspace{-.08\linewidth}
\includegraphics[width=.38\linewidth]{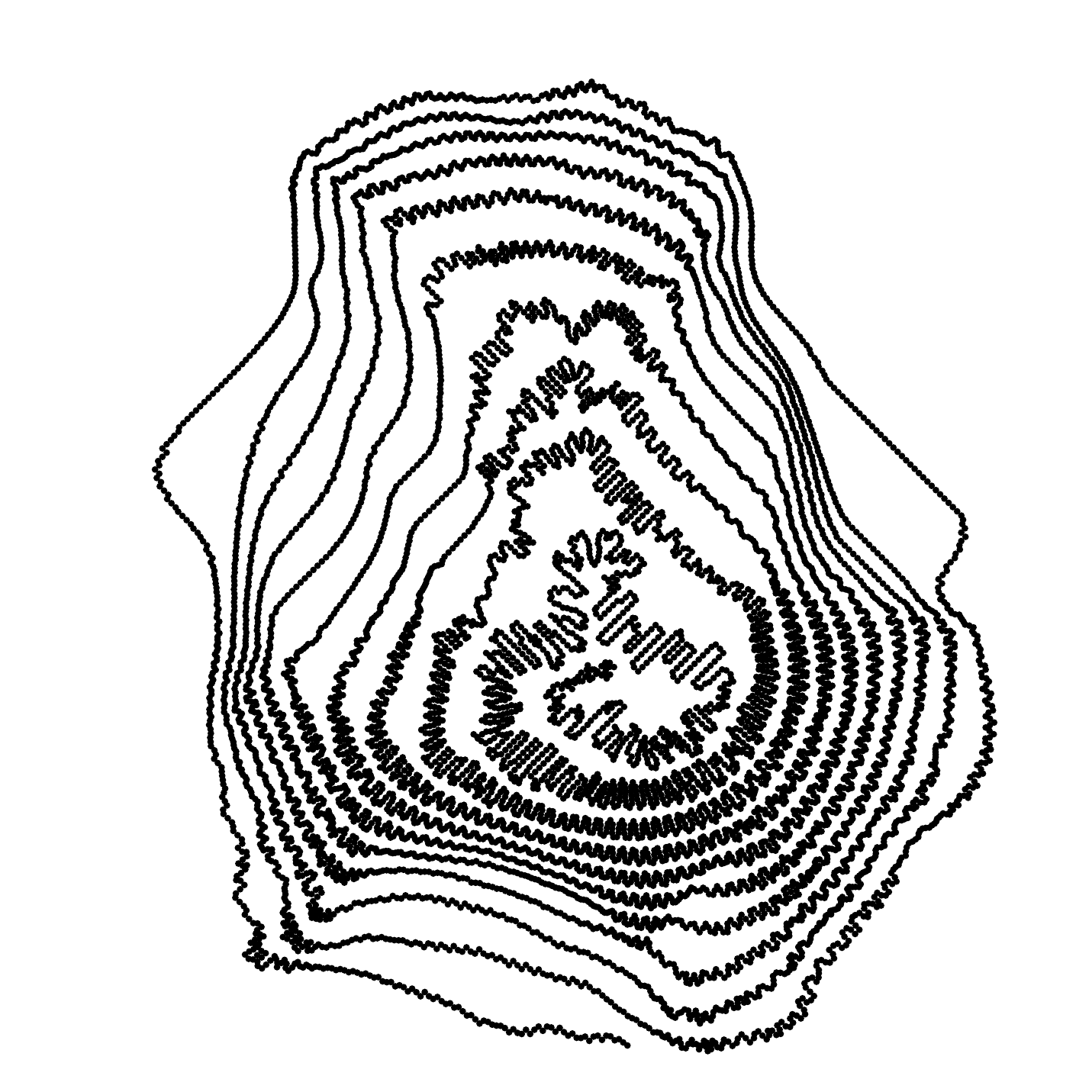}&
\hspace{-.11\linewidth}
\includegraphics[width=.38\linewidth]{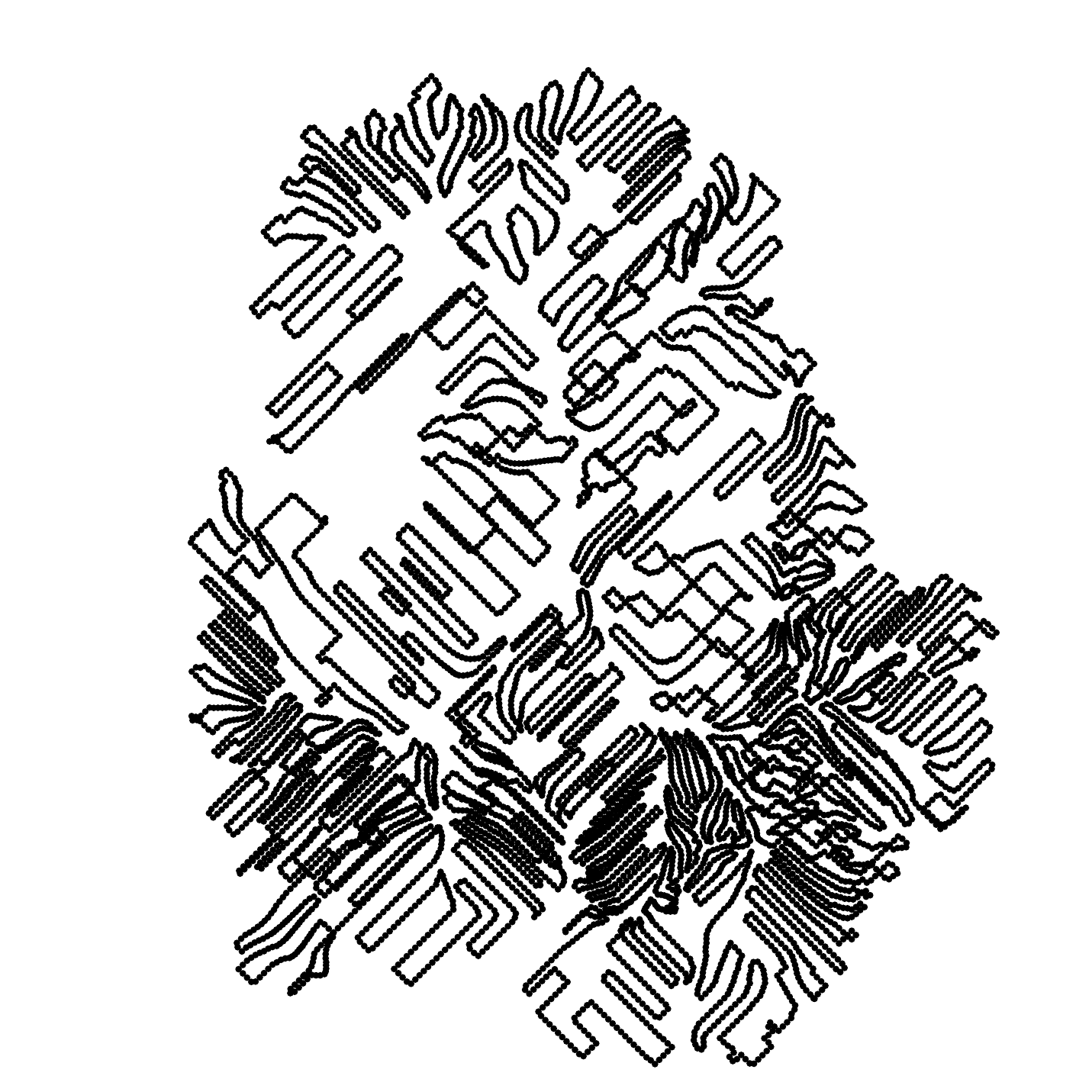}&
\hspace{-.11\linewidth}
\includegraphics[width=.38\linewidth]{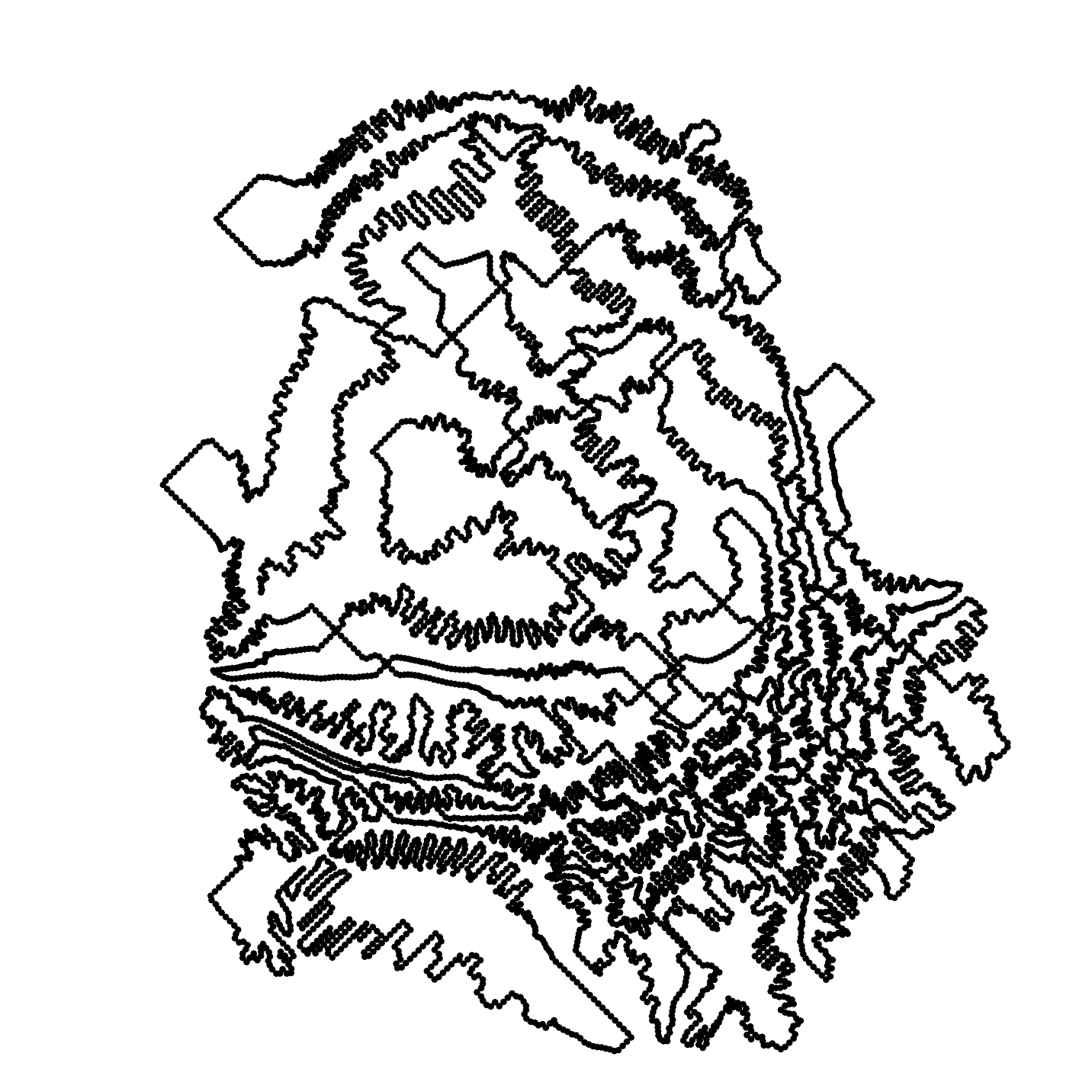}\\
\rotatebox{90}{\hspace{.15\linewidth}$s^{(1000)}$}&
\hspace{-.08\linewidth}
\includegraphics[width=.38\linewidth]{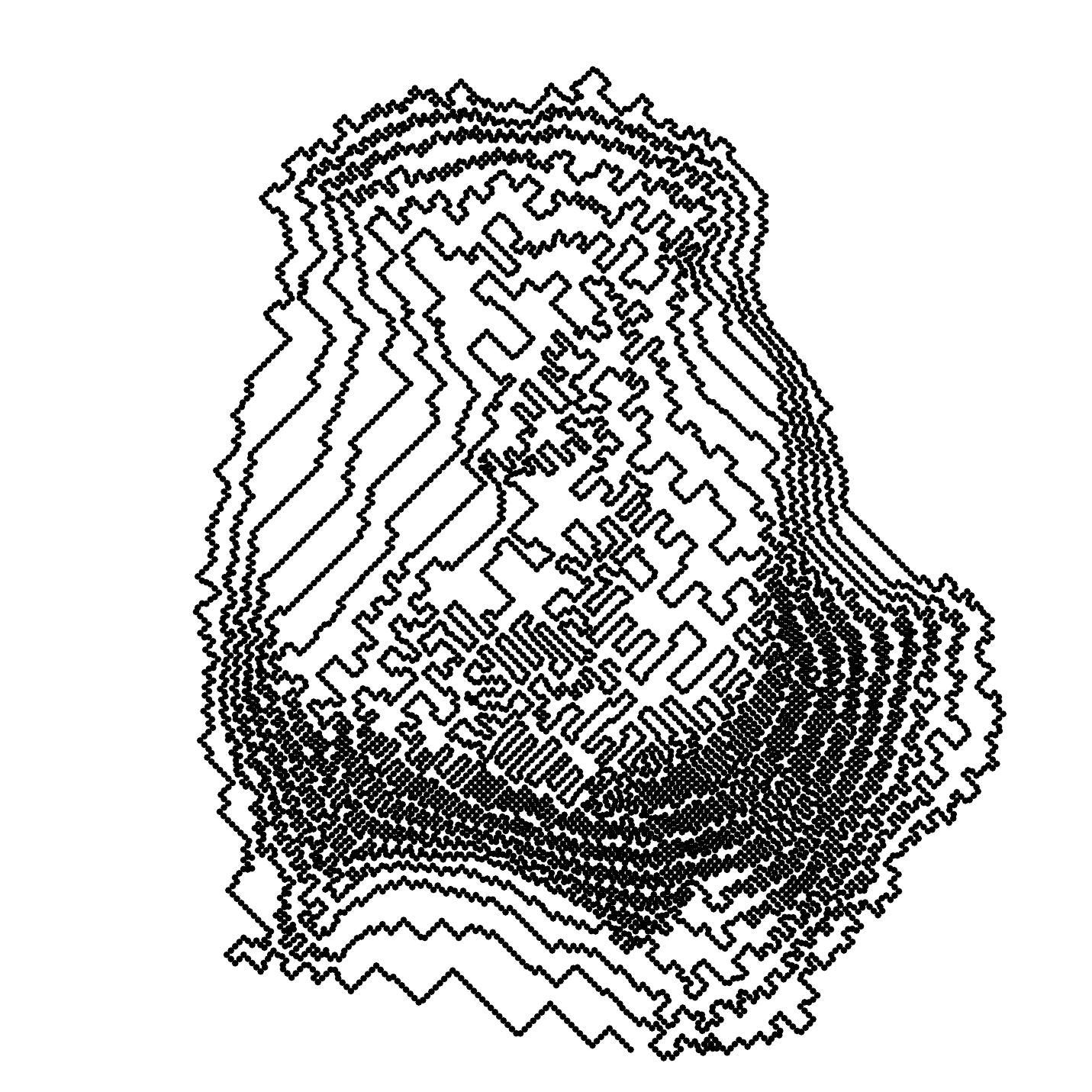}&
\hspace{-.11\linewidth}
\includegraphics[width=.38\linewidth]{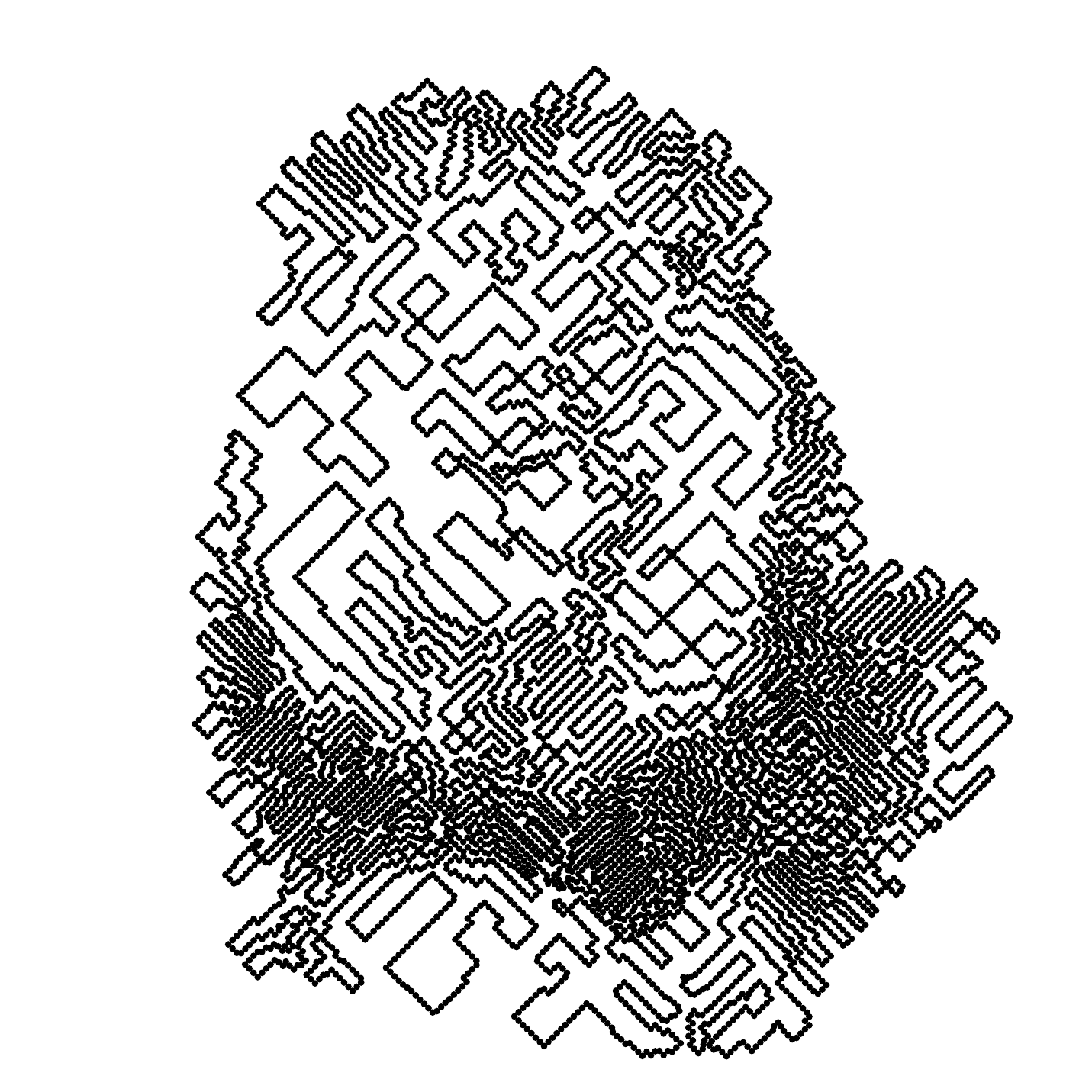}&
\hspace{-.11\linewidth}
\includegraphics[width=.38\linewidth]{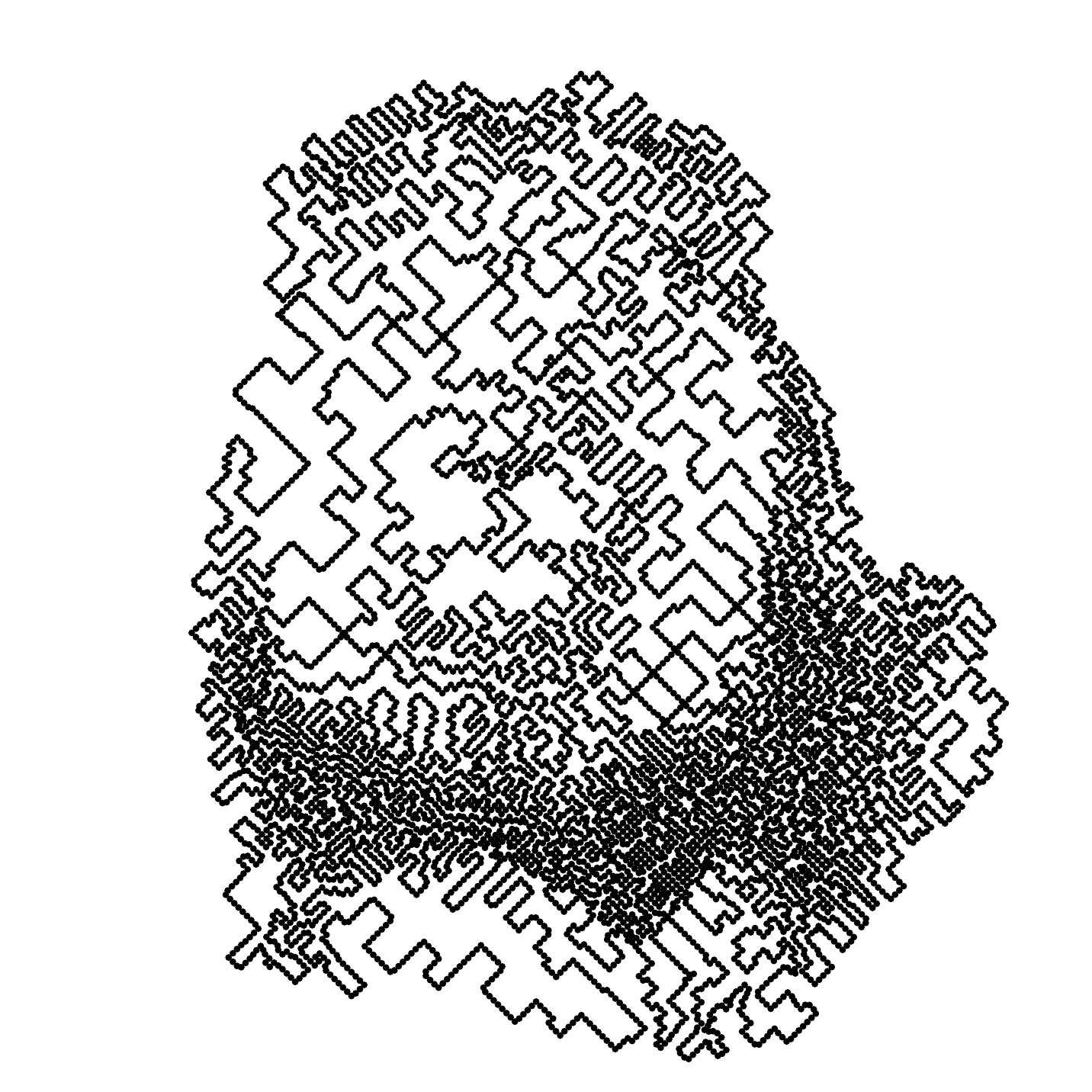}\\
\rotatebox{90}{\hspace{.15\linewidth}$s^{(30,000)}$}&
\hspace{-.08\linewidth}
\includegraphics[width=.38\linewidth]{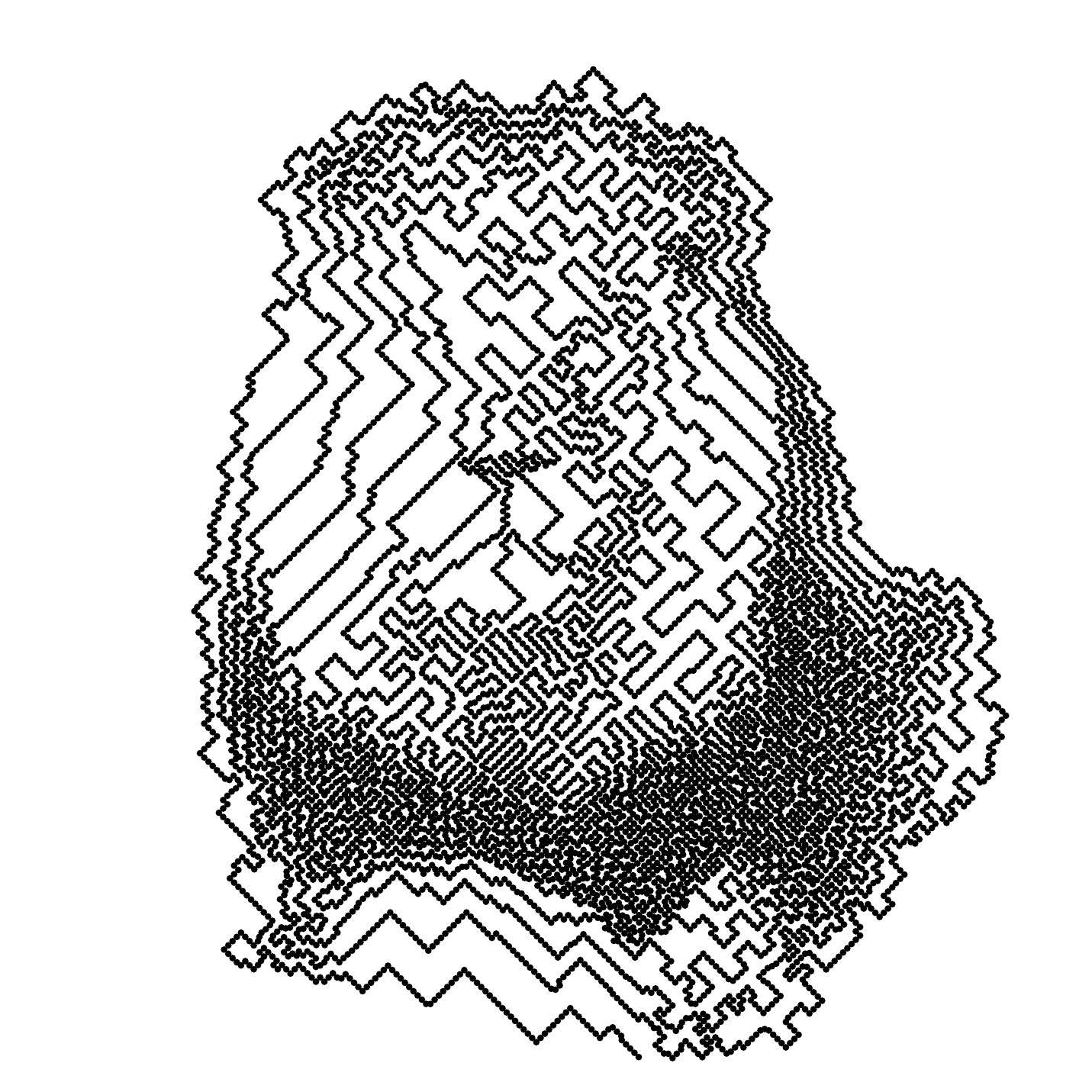}&
\hspace{-.11\linewidth}
\includegraphics[width=.38\linewidth]{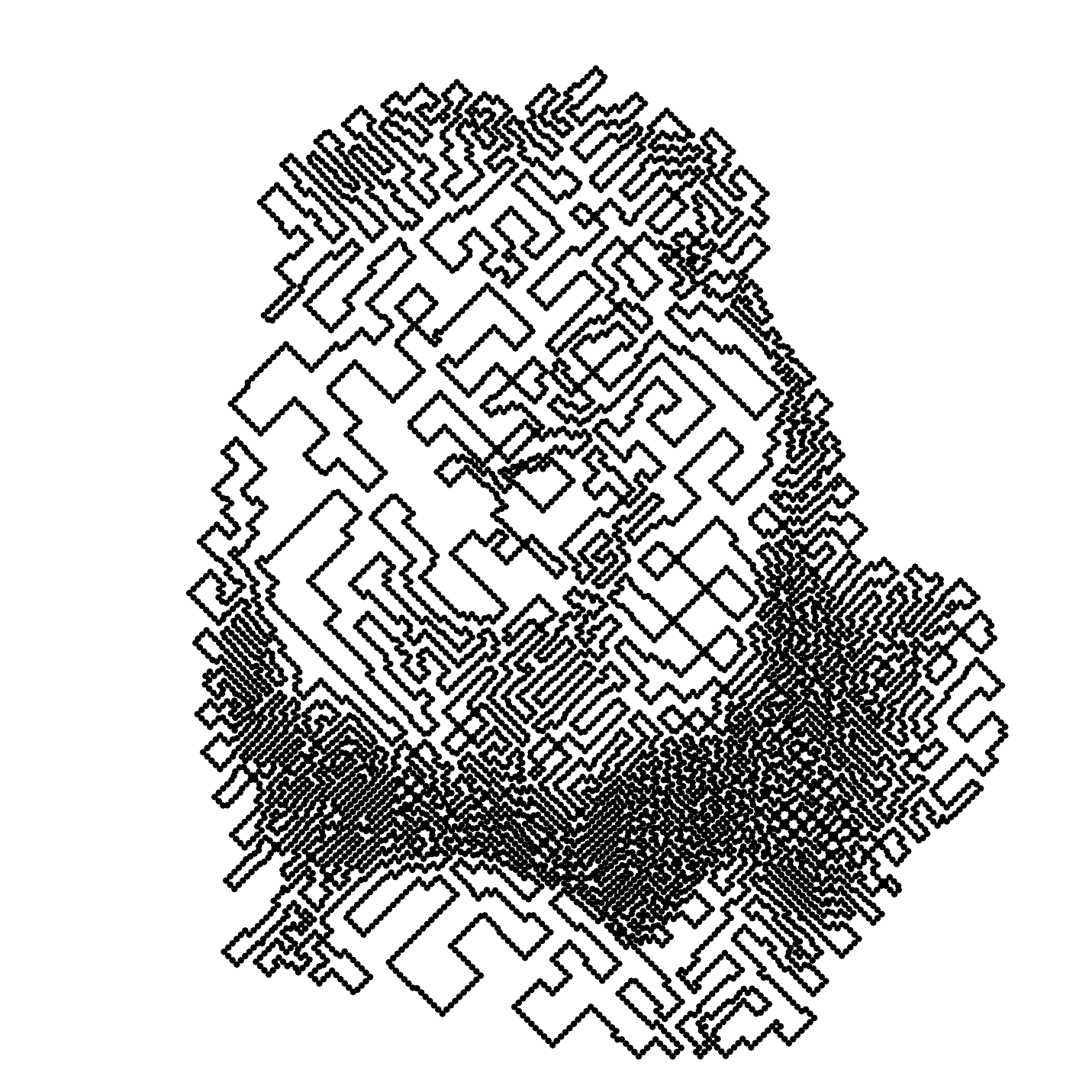}&
\hspace{-.11\linewidth}
\includegraphics[width=.38\linewidth]{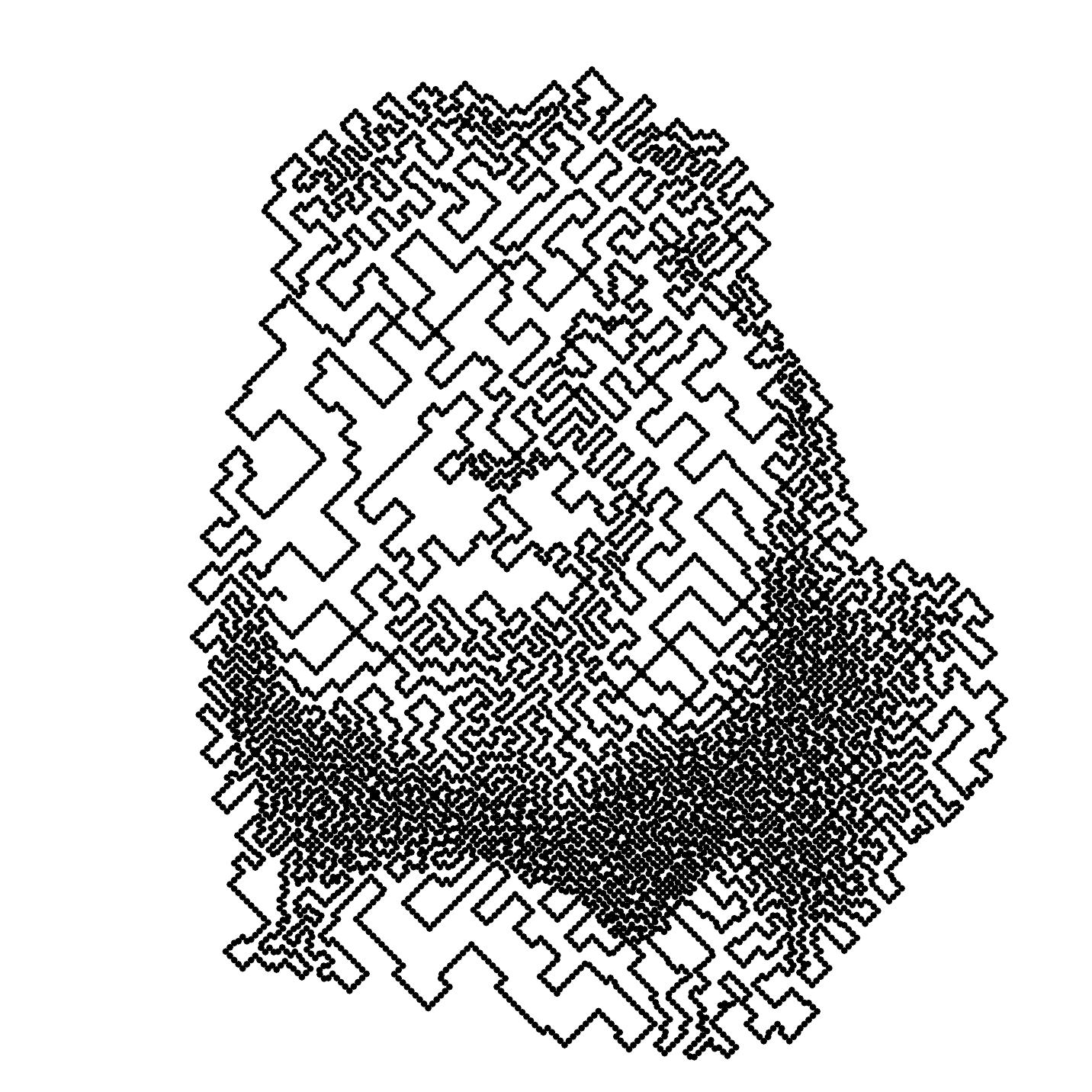}\\
\end{tabular}
\caption{\label{fig:evolution} Projection of the lion image onto $P_N^{1,\infty}$ with $N=8,000$. The figure depicts $s^{(k)}$ with several values of the iterate $k$ in Algorithm~\ref{algo}.}
\end{figure}
\begin{figure}
\begin{center}
\includegraphics[width=.7\linewidth]{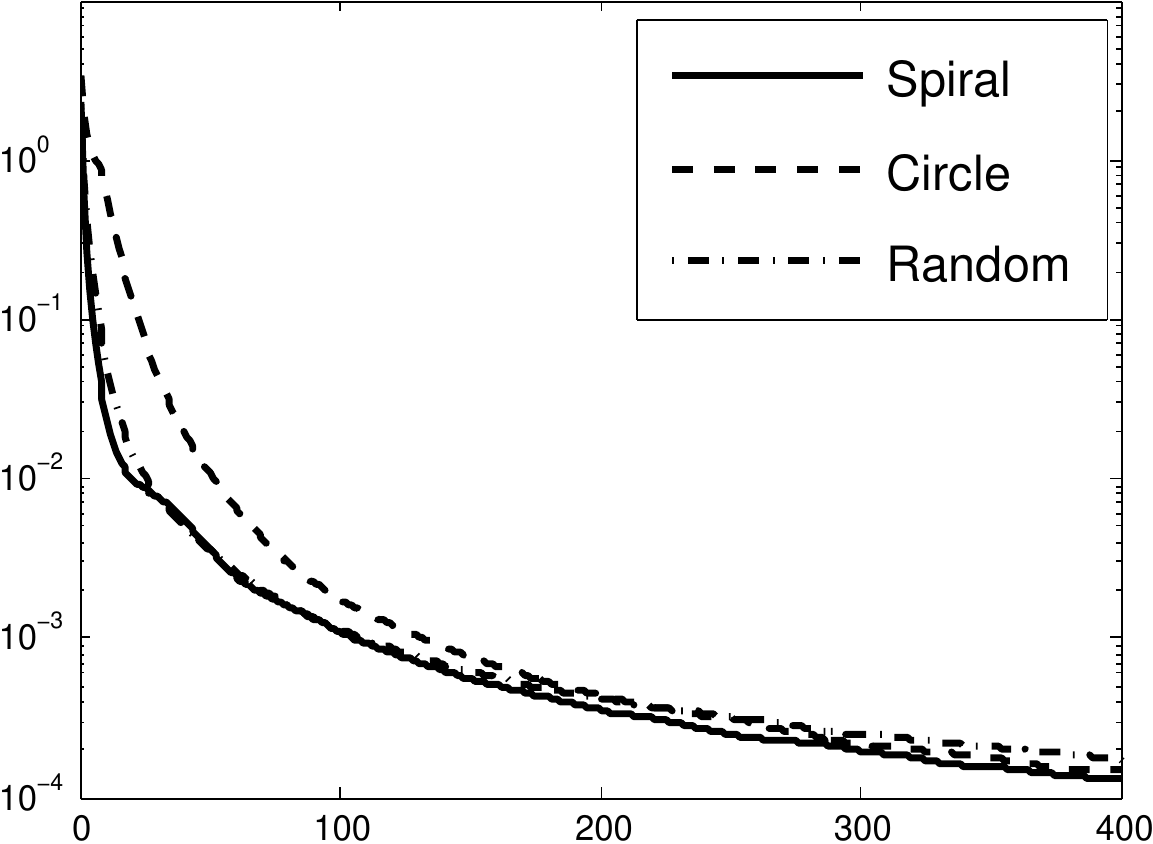}
\end{center}
\caption{\label{fig:decayCF} Decay of the cost function $J$ for the three experiments depicted in Fig.~\ref{fig:evolution}. We represent $\log_{10}(J(k)-m)$ for $k\leq 400$ where $m$ is the mimimal value of $J$ during the first $30,000$ iterations.}
\end{figure}

In Figure~\ref{fig:fille_perle_XP}, we show the projection of the famous Girl with a Pearl Earring painting, after $10,000$ iterations.  
To really see the precision of the algorithm, we advise the reader to blink the eyes or to take a printed version of the paper away. 
From a close distance, the curves or points are visible. 
From a long distance, only the painting appears.
\begin{figure}[!h]
\includegraphics[width=\linewidth]{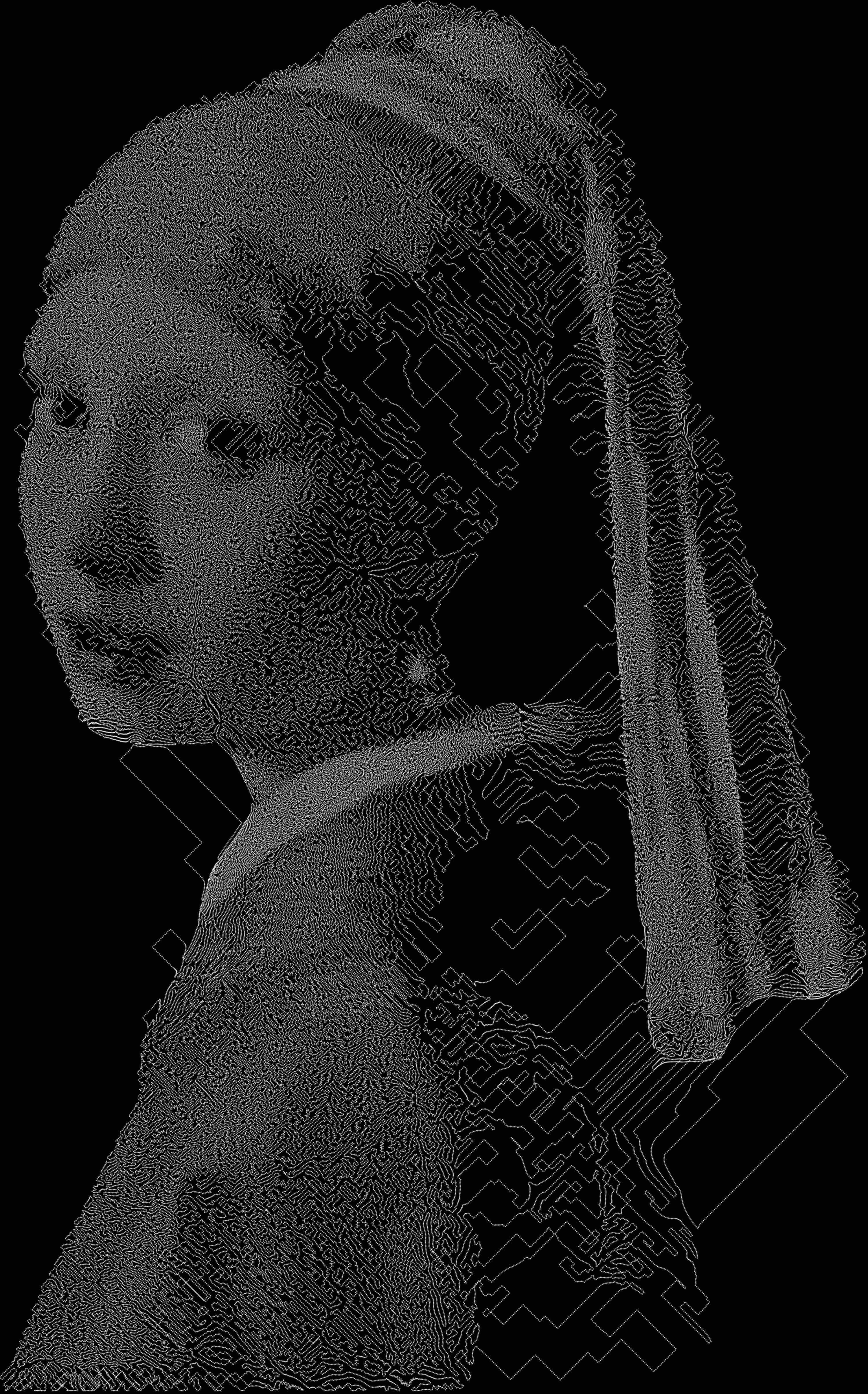}
\caption{\label{fig:fille_perle_XP} Projection of Meisje met de Parel, Vermeer 1665, onto $P_N^{1,\infty}$ with $N=150,000$. The figure depicts $s^{(10,000)}$ obtained with Algorithm~\ref{algo}.}
\end{figure}

\subsection{Projection onto $P_N^{m,q}$}
We now consider projections onto more general measure spaces, such as $\mathcal{M}\left(\mathcal{P}_T^{m,q}\right)$, in order to show that different measures spaces can be considered.
In Fig.~\ref{fig:Pmq}, we show different behaviours for different $m\in \{1,2\}$ and $q\in \{1,2,\infty\}$.
We also show a large scale example with a picture of Marylin Monroe in Figure \ref{fig:Marylin}.
\begin{figure}[!h]
\begin{center}
\begin{tabular}{cc}
$m=1, q=1$&$m=1, q=1$\\[-.01\linewidth]
\small{(small $\alpha_1$)}&\small{(large $\alpha_1$)} \\
\includegraphics[width=.45\linewidth]{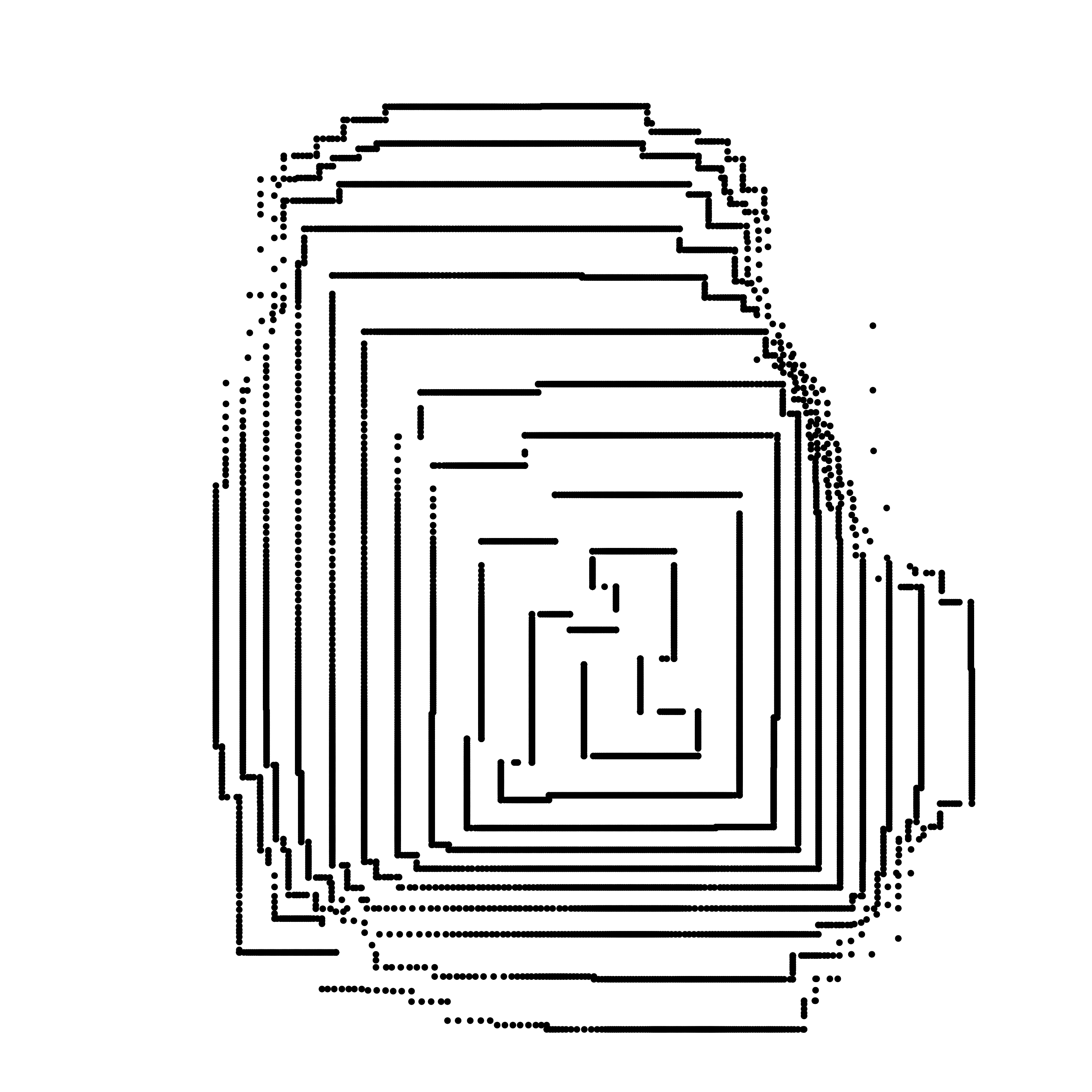}&
\includegraphics[width=.45\linewidth]{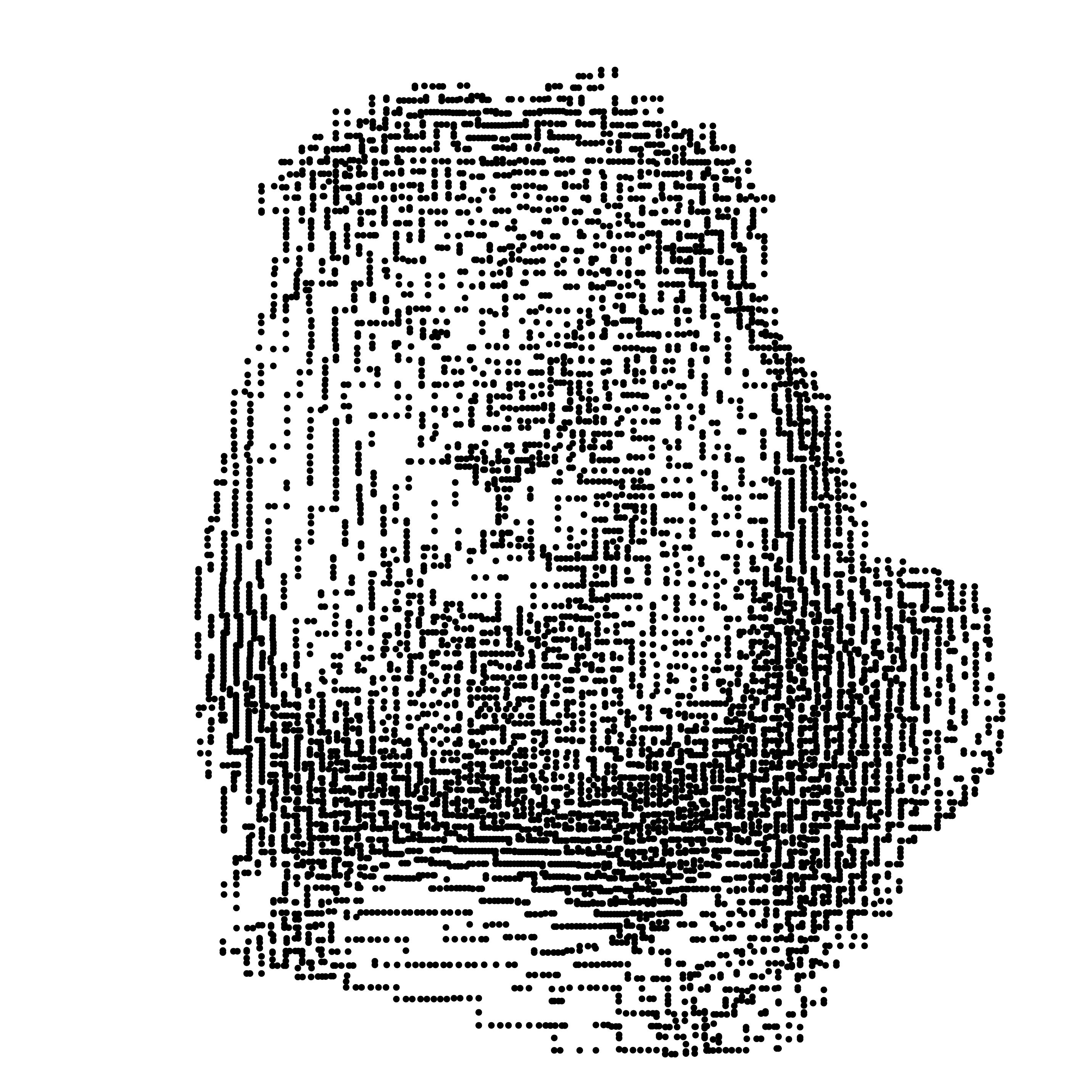}\\
$m=1, q=2$ & $m=1, q=\infty$\\
\includegraphics[width=.45\linewidth]{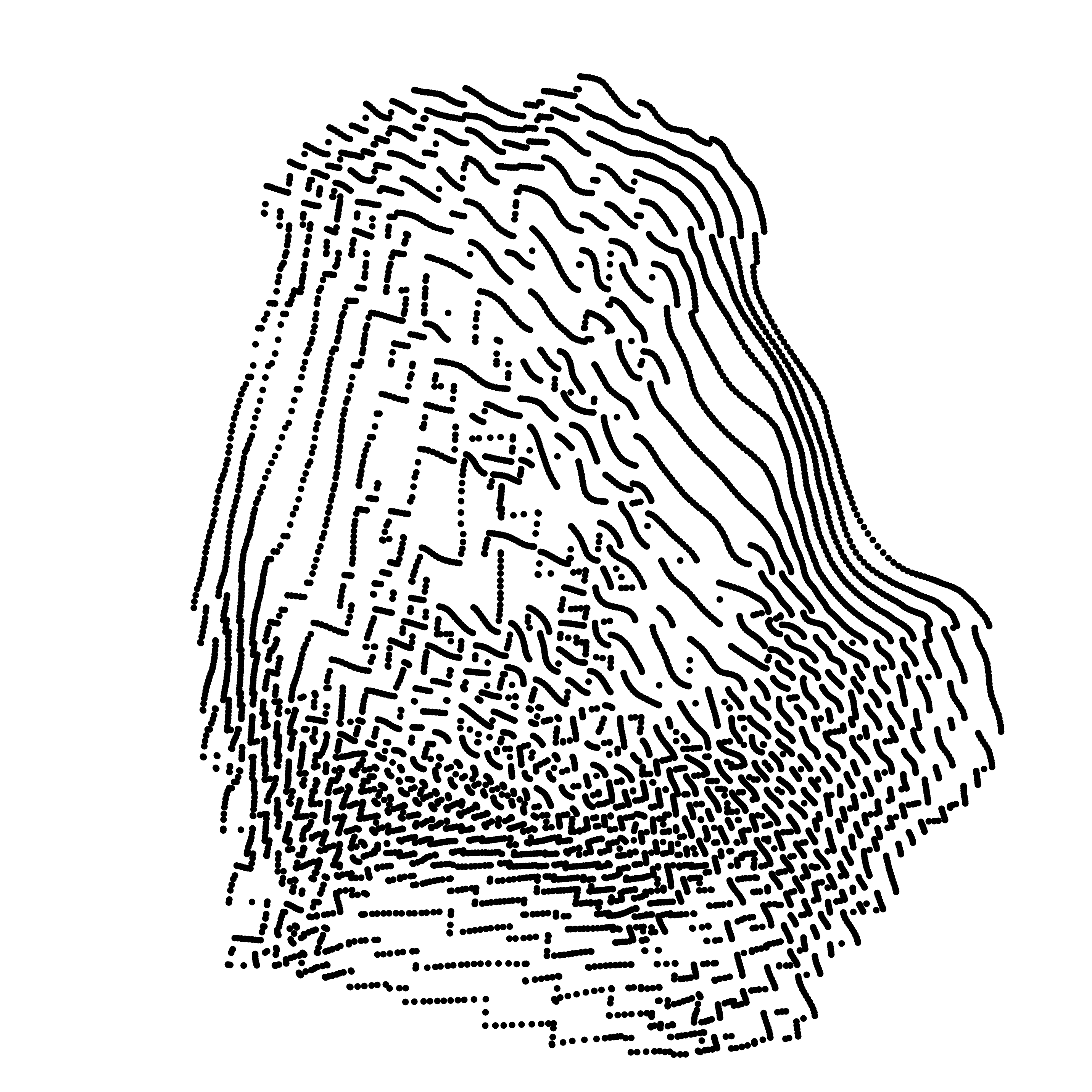}&
\includegraphics[width=.45\linewidth]{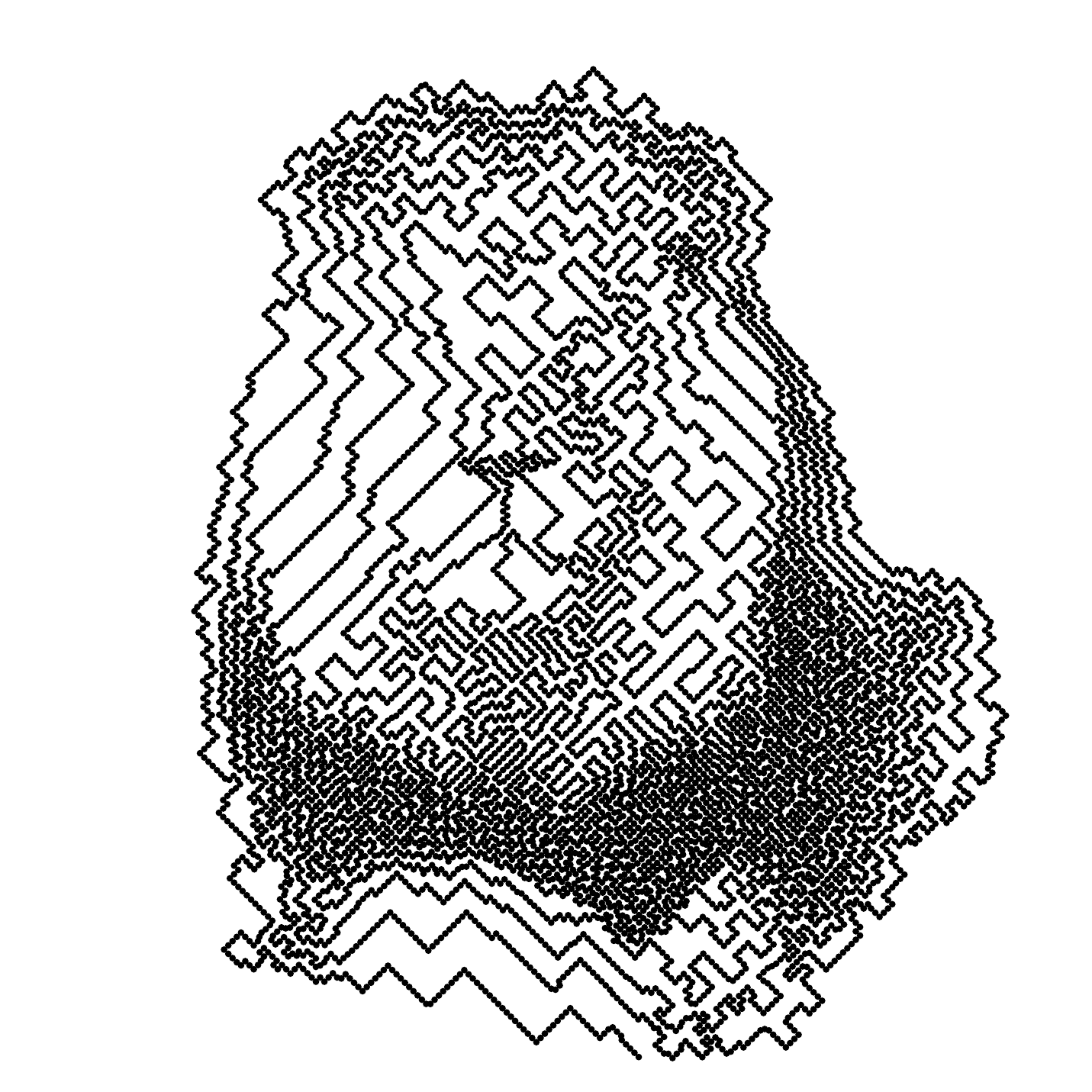}\\
$m=2, q=\infty$& $m=2, q=\infty$ \\[-.01\linewidth]
 &\small{(isotropic norm)} \\
\includegraphics[width=.45\linewidth]{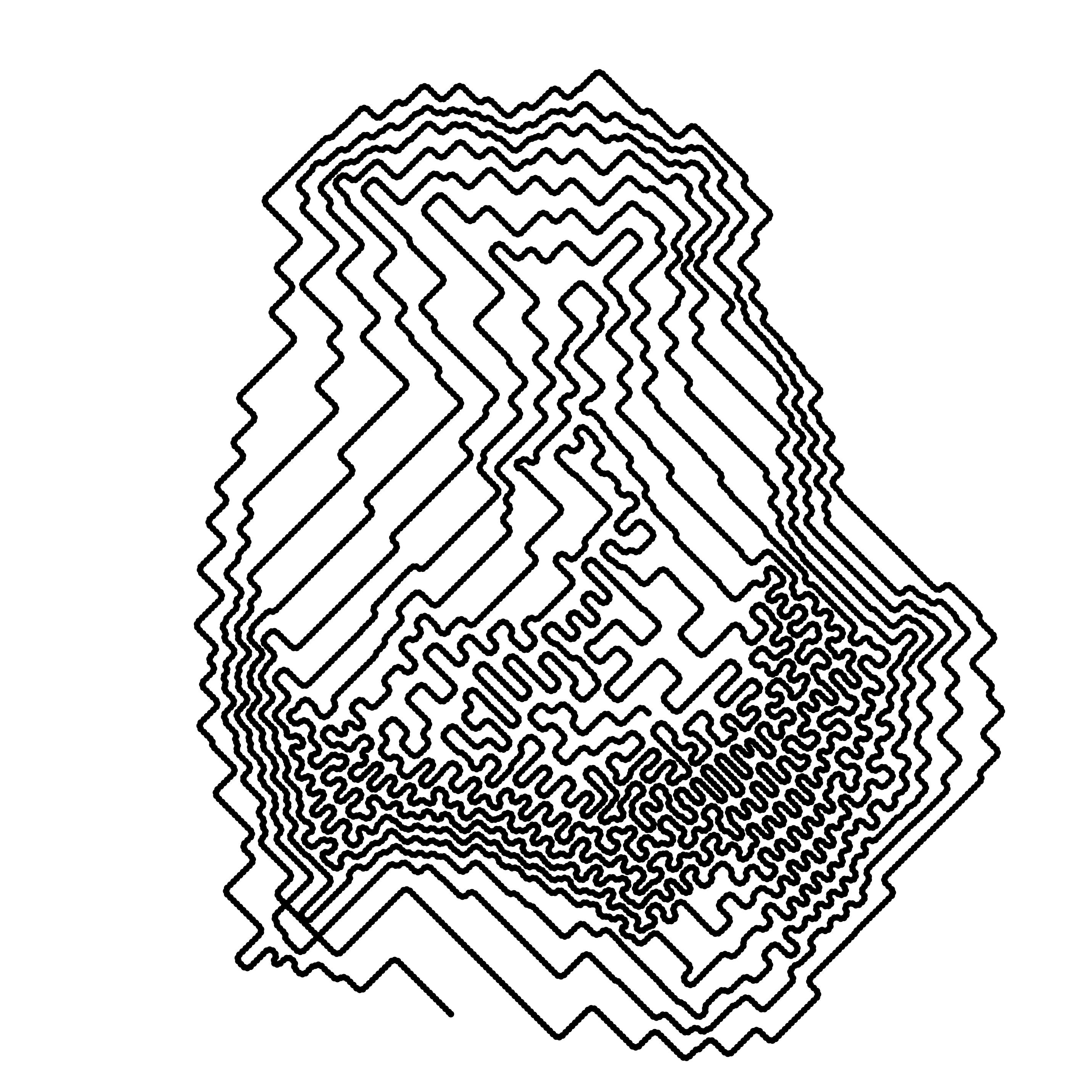}&
\includegraphics[width=.45\linewidth]{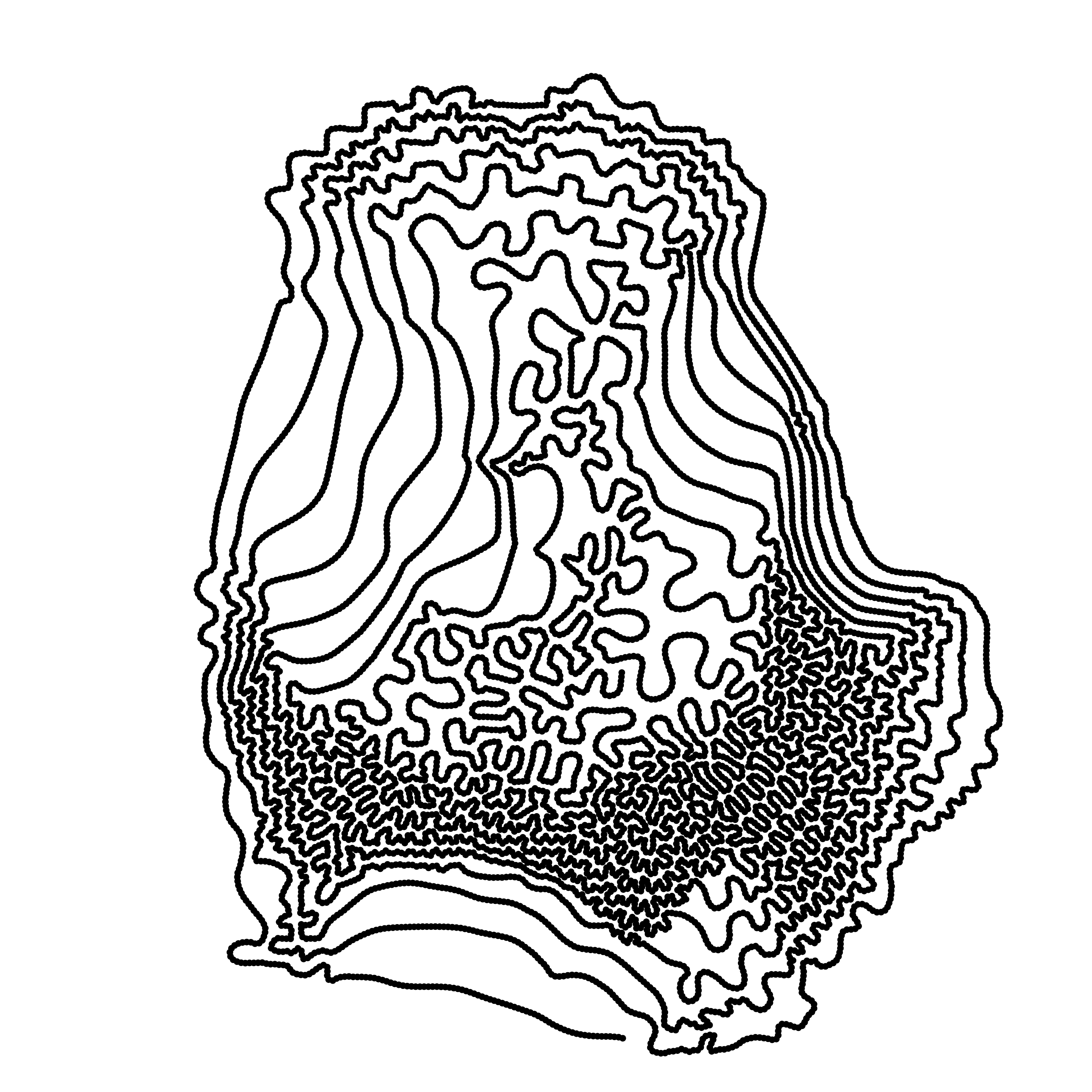}\\
\end{tabular}
\end{center}
\caption{\label{fig:Pmq} Projection of the lion image onto $P_N^{m,q}$ with $N=8,000$, and $m\in \{1,2\}$ and $q\in \{1,2,\infty\}$.}
\end{figure}

\begin{figure}[h]
\includegraphics[width=\linewidth]{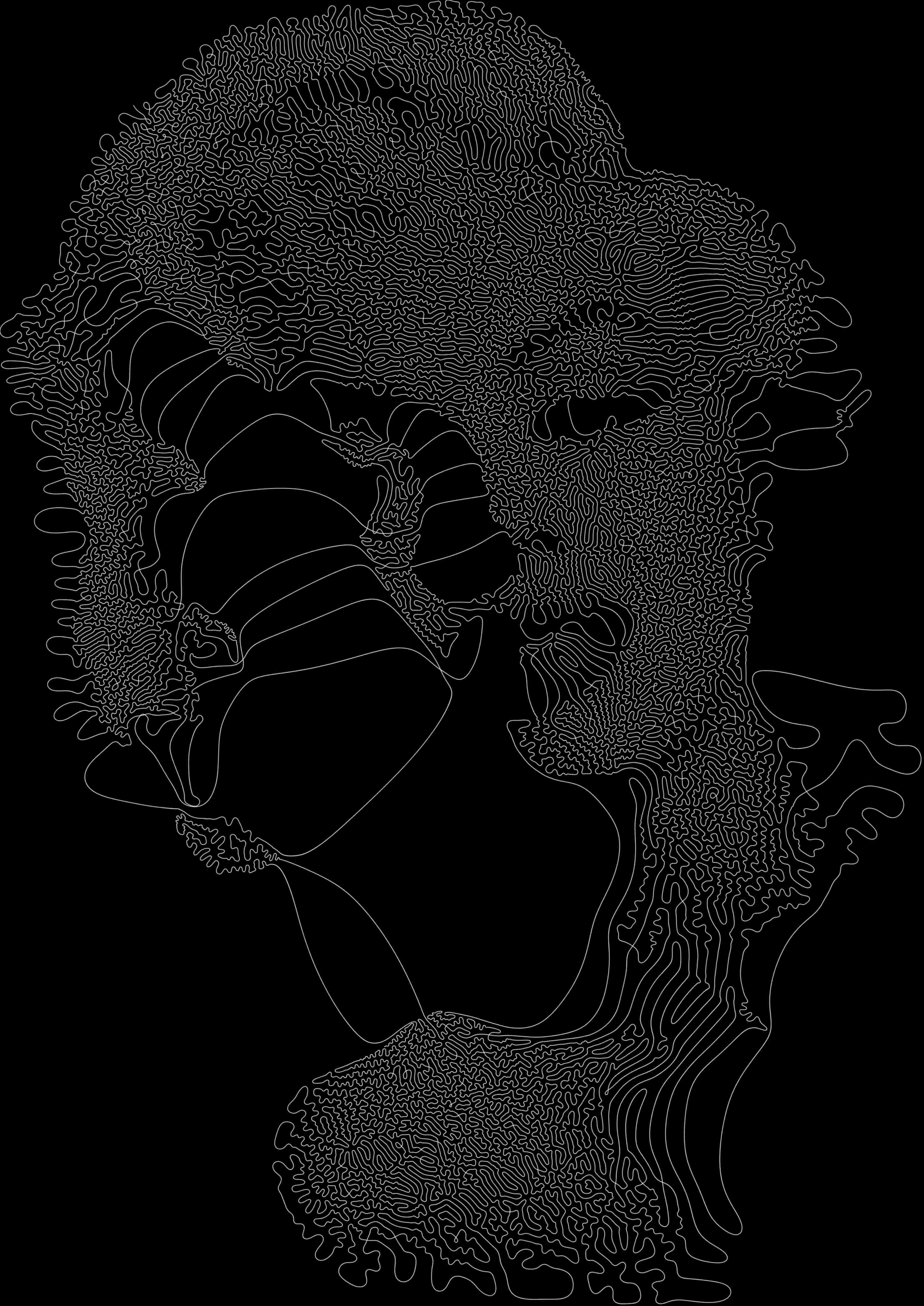}
\caption{\label{fig:Marylin} Projection of Marylin image, onto the set:\\ $\displaystyle \mathcal{C}=\{ p\in (W^{2,\infty}([0,T]))^2,  \sup_{i\in[1,N]}\left(\|D_1 p(i)\|_2 \right)\leq \alpha_1, \sup_{i\in[1,N]}\left(\|D_2 p(i)\|_2 \right)\leq \alpha_2\}$, with $N=100,000$. The figure depicts $s^{(10,000)}$ obtained with Algorithm~\ref{algo}.}
\end{figure}

%% file: Conclusion.tex
We analyzed the basic properties of a variational problem to project a target Radon measure $\pi$ on arbitrary measures sets $\mathcal{M}_N$. 
We then proposed a numerical algorithm to find approximate solutions of this problem and gave several guarantees.
An important application covered by this algorithm is the projection on the set of $N$-point measures, which is often called quantization and appears in many different areas such as finance, imaging, biology,... 
To the best of our knowledge, the extension to \textit{arbitrary} measures set is new, and opens many interesting application perspectives.
As examples in imaging, let us mention open topics such as the detection of singularities \cite{aubert2005detecting} (e.g. curves in 3D images) and sparse spike deconvolution in dimension $d$ \cite{duval2013exact}.

To finish, let us mention an important open question. 
We provided necessary and sufficient conditions on the sequence $(\mathcal{M}_N)_{N\in \N}$ for the sequence of \textit{global} minimizers $(\mu^*_N)_{N\in \N}$ to weakly converge to $\pi$. 
In practice, finding the global minimizer is impossible and we can only expect finding critical points.
One may therefore wonder whether all sequences of critical points weakly converge to $\pi$. 
An interesting perspective to answer this question is the use of mean-filed limits \cite{fornasier2013consistency}.